\documentclass[12pt]{amsart}

\usepackage{amssymb}
\usepackage[colorlinks=true]{hyperref}
\usepackage{mathrsfs}
\usepackage[all]{xy}
\usepackage{todonotes}
\setuptodonotes{fancyline,color=blue!30}

\usepackage{enumitem}
\usepackage[normalem]{ulem} %per cancellare
\usepackage{mathrsfs}

\usepackage{blkarray}
\usepackage{tikz}
\usetikzlibrary{tikzmark}

\usepackage{algorithmic}
\usepackage[ruled]{algorithm}
\usepackage{caption}

\newcommand{\supp}{\mathrm{supp}}
\newcommand{\PP}[1]{\mathcal P_{#1}}

\newcommand{\Ht}{\mathrm{Ht}}
\newcommand{\NN}{\mathcal{N}}

\newcommand{\Sr}{\mathbf{Sr}}
\newcommand{\kk}{\mathbb{K}}

\newcommand{\MFFunctor}[1]{\underline{\mathbf{Mf}}_{#1}}
\newcommand{\MFScheme}[1]{\mathbf{Mf}_{#1}}

\newcommand{\HScheme}[2]{\mathbf{Hilb}_{#1}^{#2}}

\newcommand{\rid}[1]{\longrightarrow_{#1}}
\newcommand{\crid}[1]{\longrightarrow_{#1}^+}

\newtheorem{theorem}{Theorem}[section]
\newtheorem{corollary}[theorem]{Corollary}
\newtheorem{proposition}[theorem]{Proposition}
\newtheorem{lemma}[theorem]{Lemma}

\theoremstyle{definition}
\newtheorem{definition}[theorem]{Definition}

\newtheorem{example}[theorem]{Example}
\theoremstyle{remark}
\newtheorem{remark}[theorem]{Remark}
\numberwithin{equation}{section}
\newtheorem{algo}[theorem]{Algorithm}

\title[Cohen-Macaulay, Gorenstein, complete intersection by marked bases]{Cohen-Macaulay, Gorenstein and complete intersection conditions by marked bases}
\author[C.~Bertone]{Cristina Bertone}
\address{Dipartimento di Matematica \lq\lq G.~Peano\rq\rq\\ Universit\`a degli Studi di Torino\\ 
         Via Carlo Alberto 10\\ 10123 Torino\\ Italy.}
\email{\href{mailto:cristina.bertone@unito.it}{cristina.bertone@unito.it}}
\urladdr{\url{https://sites.google.com/view/cristinabertone/}}

\author[F.~Cioffi]{Francesca Cioffi}
\address{Dip.~di Matematica e Appl. \\ Universit\`a degli Studi di Napoli Federico II\\ Via Cintia \\ 80126 Napoli \\ Italy.}
\email{\href{mailto:cioffifr@unina.it}{cioffifr@unina.it}}

\author[M.~Orth]{Matthias Orth}
\address{Institute of Mathematics\\ University of Kassel\\ 
         34109 Kassel, Germany.}
\email{\href{mailto:morth@mathematik.uni-kassel.de}{morth@mathematik.uni-kassel.de}}

\author[W.~Seiler]{Werner Seiler}
\address{Institute of Mathematics\\ University of Kassel\\ 
         34109 Kassel, Germany.}
\email{\href{mailto:seiler@mathematik.uni-kassel.de}{seiler@mathematik.uni-kassel.de}}
\urladdr{\url{http://www.mathematik.uni-kassel.de/~seiler/}}

 
\begin{document}

\begin{abstract}
Using techniques coming from the theory of marked bases, we develop new computational methods for detection and construction of Cohen-Macaulay, Gorenstein and complete intersection homogeneous polynomial ideals. Thanks to the functorial properties of marked bases, an elementary and effective proof of the openness of arithmetically Cohen-Macaulay, arithmetically Gorenstein and strict complete intersection loci in a Hilbert scheme follows, for a non-constant Hilbert polynomial.
\end{abstract}

\maketitle

\section*{Introduction}

Marked bases are special sets of generators of polynomial
ideals in enough generic position. They have nice theoretical and computational properties which are similar to those of Gr\"obner bases. However, in contrast to Gr\"obner bases, they are able to provide an open cover of  Hilbert schemes (see \cite{Quot,CMR2015} and the references therein). Among other properties, this feature has already been applied to investigate problems in Commutative Algebra and Algebraic Geometry (see \cite{Gore}, for instance).

Inspired by the article \cite{Seiler2012}, we use techniques coming from the theory of marked bases to develop constructive characterizations of Cohen-Macaulay, Gorenstein and complete intersection homogeneous polynomial ideals by conditions on the coefficients of the polynomials of the marked bases of their Artinian reductions. These conditions even characterize the coefficients of the polynomials of the marked bases we consider due to the procedure that is used to obtain the Artinian reductions (see Proposition \ref{prop:Artinian reduction}).

Among the important roles of these ideals in several contexts, it is relevant that they satisfy some openness conditions, such as the Cohen-Macaulay locus and the Gorenstein locus in a Hilbert scheme are open subsets (\cite[Théorème~(12.2.1)(vii)]{Gro} and~\cite{Stoia}) and the family of (strict) complete intersection curves in $\mathbb P^3$ is an open subset which may not be closed (for example, see \cite[Exercises~1.3 and~1.4]{Hart-Def}). Moreover, the Nagata criterion holds for Gorenstein and complete intersection properties in rings (see \cite{GM} and \cite{Kimura} for a module version). 

By means of our constructive characterizations, we obtain an elementary proof of the openness of the three loci of points in a Hilbert scheme corresponding to either arithmetically Cohen-Macaulay schemes or arithmetically Gorenstein schemes or strict complete intersection schemes, for every non-constant Hilbert polynomial, also providing an explicit representation via suitable equations and inequalities. Up to our knowledge, in this generality, these results are a novelty.

For a presentation of Cohen-Macaulay, Gorenstein and complete intersection rings and corresponding closed projective schemes, we refer to \cite{Hart-book, Bass, Huneke, Kunz, IK, Ei, Mi}.
%and \cite[Section 4]{Mi}. 

Concerning constructive approaches to the study of these objects, general structure theorems of Gorenstein ideals are given in \cite{BE-Gore} in codimension~$3$ and discussed in \cite{Reid2015} for a generalization to codimension~$4$. Some explicit constructions of Gorenstein ideals are given in specific cases (for example, see \cite{NS1,NS2,Elias-Rossi}). See also \cite{AbCo,SaSt} and the references therein for some recent related interesting investigations and open problems. 
Nevertheless, in an affine framework, a study of general constructions of Gorenstein and complete intersection ideals can be found in~\cite{KLR2020,KLR2022} by means of properties of border bases and border schemes, whose relations with marked bases and marked schemes are investigated in \cite{BC2022}. 

In our paper, we develop general constructions in every codimension and every polynomial ring on a field in terms of marked bases. Here is an outline of the content. 

After recalling the main features of marked bases and marked schemes and their relations with Hilbert schemes, which we will refer to when necessary (see Sections~\ref{sec:preliminaries} and \ref{sec:marked functors}), the first relevant step consists in observing that, up to a deterministic linear change of variables, marked schemes over quasi-stable ideals that are Cohen-Macaulay parameterize all the Cohen-Macaulay homogeneous polynomial ideals (see Theorem \ref{th:CM} and Corollary~\ref{cor:CM}). We also provide a method to recognize Cohen-Macaulay ideals among those having a marked basis over a truncation of any saturated quasi-stable ideal (see Algorithm~\ref{algo 1} and Theorem~\ref{th:saturation}).

In a marked scheme over a Cohen-Macaulay quasi-stable ideal, a Gorenstein ideal can be recognized by the dimension of the socle of its Artinian reduction even in terms of marked bases. Hence, thanks to Theorem~\ref{th:socle}, we obtain a new constructive method for Gorenstein homogeneous ideals by a characterization of Artinian Gorenstein homogeneous ideals in terms of the shape of their marked bases (see Corollary~\ref{cor:condition Gore}). Moreover, for a non-constant Hilbert polynomial, the openness of the arithmetically Gorenstein locus in a Hilbert scheme follows (see Corollary~\ref{cor:open Gore}). 

Regarding complete intersection ideals, we first focus on the process of distinguishing such ideals using the expected number of their minimal generators. We propose a solution that performs minimization of marked bases of homogeneous Artinian ideals by linear algebra only (see Section \ref{sec: CI conditions}). This minimization procedure is developed using the notion of border basis in the homogeneous framework of our paper. Like for the arithmetically Gorenstein property, an explicit characterization of complete intersection Artinian homogeneous ideals follows (see Corollary \ref{cor: CI}), together with an elementary proof of the openness of the strict complete intersection locus in a Hilbert scheme, for a non-constant Hilbert polynomial (see Corollary~\ref{cor:open CI}). 
Then, we focus on the more general task to construct a regular sequence contained in a given polynomial ideal. For this task we obtain a qualitative answer, adapting to marked bases a result of Eisenbud and Sturmfels, which has been developed in~\cite{EiStu1994} for Gr\"obner bases (see Theorem~\ref{theorem:EiStu theorem}).

Throughout the paper, we recall the definitions that are needed and give examples and applications of the computational methods that arise from our results.

\section{Preliminaries on marked bases}
\label{sec:preliminaries}
 
Let $R:=\kk[x_0,\dots,x_n]$ be the polynomial ring over a field $\kk$ in $n+1$ variables ordered as $x_0<x_1<\dots<x_n$,
and $\mathbb T=\{x_0^{\alpha_0}x_1^{\alpha_1}\dots x_n^{\alpha_n} : (\alpha_0,\dots,\alpha_n)\in \mathbb Z^n_{\geq 0}\}$ be the set of its terms. 
For every term $x^\alpha:=x_0^{\alpha_0}x_1^{\alpha_1}\dots x_n^{\alpha_n}\not=1$ we denote by $\min(x^{\alpha}):=\min_{i=0,\dots,n}\{x_i :  \alpha_i\not=0 \}$ the minimum variable that appears in $x^{\alpha}$ with a non-null exponent. Analogously, we set $\max(x^\alpha):=\max_{i=0,\dots,n}\{x_i :  \alpha_i\not=0 \}$. Given a term $x^\alpha$, a variable $x_i$ is called a {\em multiplicative variable} of $x^\alpha$ if $x_i\leq \min(x^\alpha)$, otherwise it is called a {\em non-multiplicative variable} of $x^\alpha$.

An ideal $J$ is said a {\em monomial ideal} if it is generated by terms. The minimum set of generators of a monomial ideal $J$ made of terms is denoted by $B_J$. The {\em sous-escalier of $J$} is the set $\mathcal N(J)$ made of the terms outside $J$.

For every Noetherian $\kk$-algebra $A$, we set $R_A:=A\otimes R=A[x_0,\dots,x_n]$ and consider the standard grading for which $\deg(x_i)=1$, for every $i\in \{0,\dots,n\}$, and $\deg(a)=0$ for every $a\in A$. The degree of a term $x^\alpha$ is $\vert \alpha \vert=\sum_i \alpha_i$.

For every set $N$ of homogeneous polynomials in $R_A$, we denote by $(N)$ the ideal generated by $N$ and by $\langle N\rangle_A$ the $A$-module generated by $N$ over $A$. Moreover, for every integer $t$, we denote by $N_{\geq t}$  the set of the homogeneous polynomials of degree $\geq t$ of $N$ and by $N_t$ the set of the homogeneous polynomials of $N$ of degree $t$. 

For a homogeneous ideal $I\subset R_A$, we continue to write $I_{\geq t}$ even to denote the ideal $(I_{\geq t})$ and $I_t$ to denote the $A$-module $\langle I_t\rangle_A$. The Hilbert function $H_{R_A/I}$ is the function $H_{R_A/I}: \mathbb Z \to \mathbb Z$ such that $H_{R_A/I}(t)$ is the number of generators of a $A$-basis of $(R_A/I)_t$, being $(R_A/I)$ a free module. For $t\gg 0$, $H_{R_A/I}(t)$ assumes the same value of a numerical polynomial $p(z)$ that is called Hilbert polynomial. 

A monomial ideal $J$ is said {\em quasi-stable} if for every term $x^\tau \in J$ and every non-multiplicative variable $x_k>\min(x^\tau)$ of $x^\tau$ there exists an exponent $s_k$ such that the term $\frac{x^\tau}{\min(x^\tau)}x_k^{s_k}$ belongs to $J$.

A monomial ideal $J$ is quasi-stable if and only if there is a (unique) finite set of generators $\mathcal P_J$ of $J$ made of terms such that, for every term $x^\tau\in J\setminus \mathcal P_J$, there exists a unique term $x^\sigma \in \mathcal P_J$ so that $x^\tau=x^\delta x^\sigma$ with $\max(x^\delta) \leq \min(x^\sigma)$. The set $\mathcal P_J$, which is called the {\em Pommaret basis of $J$}, contains $B_J$. When $\mathcal P_J$ is equal to $B_J$, the ideal $J$ is said a {\em stable} ideal. Every Artinian monomial ideal is quasi-stable.

For any homogeneous ideal $I\subseteq R$ we denote by $\mathrm{sat}(I)$ its satiety, which is the minimum integer $s$ such that $I_s$ is equal to the homogeneous part of degree $s$ of the saturation $I^{sat}:=\{f\in R \vert \ \forall i\in \{0,\dots,n\} \ \exists k_i\in \mathbb N : x_i^{k_i}f \in I \}$ of $I$. The satiety of a quasi-stable ideal $J$ coincides with the maximum degree of a term  divisible by the last variable in the Pommaret basis of $J$.

A {\em marked polynomial $f$} is a polynomial together with a given term $x^\alpha$ that appears in $f$ with coefficient equal to $1_A$ and which is said {\em head term of $f$} and denoted by $\mathrm{Ht}(f)$. Usually we will write $f_\alpha$ to denote a marked polynomial with head term equal to $x^\alpha$.

\begin{definition}\label{def:sets and bases}
A {\em $\mathcal P_J$-marked set} $\mathcal H=\{h_\alpha\}_{x^\alpha \in \mathcal P_J}$ is a set of homogeneous marked polynomials such that, for every term $x^\alpha \in \mathcal P_J$, there exists a unique polynomial $h_\alpha\in \mathcal H$ such that every term other than $\mathrm{Ht}(f)=x^\alpha$ that appears in $h_\alpha$ with a non-null coefficient belongs to $\mathcal N(J)$. A $\mathcal P_J$-marked set $\mathcal H$ is said a {\em $\mathcal P_J$-marked basis} if $(R_A)_t=(\mathcal H)_t \oplus \langle \mathcal N(J)_t\rangle_A$, for every degree $t$.
\end{definition}

When we say that an ideal $I$ has a marked set (resp.~basis)  over a quasi stable ideal~$J$ we mean that $I$ is generated by a $\mathcal P_J$-marked set (resp.~basis).

For every $\mathcal P_J$-marked set $\mathcal H$ and for every integer $t$ we consider 
\begin{equation}\label{eq:multipli}
\mathcal H^{(t)}:=\{x^\delta h_\alpha : h_\alpha \in \mathcal H, x^\delta=1 \text{ or } \max(x^\delta)\leq \min(x^\alpha), \deg(x^\delta x^\alpha)=t\}.
\end{equation}
If a polynomial $x^\delta h_\alpha$ belongs to $\mathcal H^{(t)}$ we say that $x^\delta x^\alpha$ is its head term.

\begin{definition}\label{def:riscrittura}
Given a $\mathcal P_J$-marked set $\mathcal H=\{h_\alpha\}_{x^\alpha \in \mathcal P_J}$, for every integer $t$ we denote by $\rid{\mathcal H^{(t)}}$ the transitive closure of the relation $f \rid{\mathcal H^{(t)}} f-\lambda x^\delta h_\alpha$, where $f$ is a polynomial, $x^\delta h_\alpha$ belongs to $\mathcal H^{(t)}$ and $x^\delta x^\alpha$ is a term that appears in $f$ with coefficient $\lambda$. 
We will write $f \crid{\mathcal H^{(t)}} g$ if $f \rid{\mathcal H^{(t)}} g$ and $g \in \langle \mathcal N(J) \rangle_A$. 
\end{definition}

The relation $\rid{\mathcal H^{(t)}}$ gives rise to a rewriting procedure that, for every polynomial~$f$, provides the following unique {\em standard representation} 
\begin{equation}\label{eq:standard representation}
f=\sum_{h\in \mathcal H} P_h h +g,
\end{equation}
where $P_h$ is a linear combination of terms made of powers of multiplicative variables of $\mathrm{Ht}(h)$ and $g$ belongs to $\langle \mathcal N(J)\rangle_A$ (see \cite[Proposition 4.11]{Quot}). The polynomial $g$ is denoted by $\mathrm{Rf}_I(f)$ and called the {\em reduced form of $f$ by $I$}. 

By Definition \ref{def:sets and bases} a $\mathcal P_J$-marked set $\mathcal H$ is a $\mathcal P_J$-marked basis if and only if for every polynomial $f$ there is a unique polynomial $p\in I$ such that $f-p \in \langle \mathcal N(J)\rangle_A$. In this case the reduced form of $f$ is called the {\em normal form of $f$ by $I=(\mathcal H)$} and denoted by $\mathrm{Nf}_I(f)$.

It will be useful to collect the \lq\lq coefficient polynomials\rq\rq $P_h$ of a standard representation~\eqref{eq:standard representation} in a vector that we denote by $\Sr(f-g)$, given an order for the polynomials in $\mathcal H$.

The relation $\rid{\mathcal H^{(t)}}$ is Noetherian, and even confluent (see \cite[Propositions 4.8 and~4.11]{Quot}) thanks to the uniqueness of standard representations. This property also implies the additivity of $\rid{\mathcal H^{(t)}}$, for which~$\Sr(f+f')= \Sr(f)+\Sr(f')$, for every two polynomials $f,f'$.

\begin{theorem} \cite[Corollary 4.15 and Theorem 4.18]{Quot} \label{th:equivalent conditions}
Let $\mathcal H=\{h_\alpha\}_{x^\alpha \in \mathcal P_J}$ be a $\mathcal P_J$-marked set. The following conditions are equivalent:
\begin{itemize}
\item[(i)] $\mathcal H$ is a $\mathcal P_J$-marked basis.

\item[(ii)] $(\mathcal H)_t = \langle \mathcal H^{(t)}\rangle_A$, for every $t\leq \mathrm{reg}(J)+1$.

\item[(iii)] $I_t \cap \langle \mathcal N(J)_t\rangle_A=\{0\}$, for every $t\leq \mathrm{reg}(J)+1$.

\item[(iv)] $x_i h_\alpha \crid{\mathcal H^{(t)}} 0$, for every $h_\alpha\in \mathcal H$, $x_i>\min(x^\alpha)$ and $\deg(x_ix^\alpha)=t$.
\end{itemize}
\end{theorem}

A {\em fundamental syzygy} $S_{\alpha,i}$ of a $\mathcal P_J$-marked basis $\mathcal H$ is a syzygy obtained by rewriting a polynomial $x_ih_{\alpha}$ using the procedure of Definition \ref{def:riscrittura}, where $h_\alpha$ belongs to $\mathcal H$ and $x_i$ is a non-multiplicative variable of the head term $x^\alpha$ of $h_\alpha$. The components of $S_{\alpha,i}$ are the coefficients $P_h$ from the standard representation $x_ih_{\alpha}=\sum P_h h$ guaranteed by Theorem \ref{th:equivalent conditions}{\em (iv)} and Formula \eqref{eq:standard representation}.

It is noteworthy that the set of the fundamental syzygies generate the module of syzygies of $\mathcal H$  \cite[Theorem 6.5]{Quot}. Hence, a polynomial $h_\beta\in \mathcal H$  depends on $\mathcal H\setminus \{h_\beta\}$ if and only if there exists a fundamental syzygy of $\mathcal H$ with a constant non-null element corresponding to the polynomial $h_\beta$.

The following result is a generalization of \cite[Corollary~2.3]{CR} to quasi-stable ideals.

\begin{proposition}\label{prop:known facts}
Let $I$ be the ideal generated by a $\mathcal P_J$-marked set $\mathcal H\subseteq R$.
\begin{itemize}
\item[(i)] The codimension of $I$ is higher than or equal to the codimension of $J$.
\item[(ii)] If $\mathcal H$ is a $\mathcal P_J$-marked basis, then the codimension of $I$ is equal to the codimension of $J$.
\end{itemize}
\end{proposition}

\begin{proof}
For item (i), by the standard representation \eqref{eq:standard representation} we have $R=I+ \langle\mathcal N(J)\rangle_K$, hence $\dim_K I_s\geq \dim_K J_s$, for every $s\geq 0$ and the degree of the Hilbert polynomial of $R/I$ is lower than or equal to the degree of the Hilbert polynomial of $R/J$.

For item (ii), it is enough to observe that by the definition of marked basis we have $R=I\oplus \langle \mathcal N(J)\rangle_K$.
\end{proof}

\section{Marked functor and Hilbert scheme} 
\label{sec:marked functors}

The set of ideals $I$ having a $\mathcal P_J$-marked basis is called {\em the $\mathcal P_J$-marked family} and can be parametrised by an affine scheme which represents a functor from the category of Noetherian $\kk$-Algebras to that of Sets.
We briefly recall the definition of this functor and the construction of the representing affine scheme.
 
The {\em marked functor} from the category of Noetherian $\kk$-algebras to the category of sets
\[
\MFFunctor{J}: \underline{\text{Noeth}\ \kk\!\!-\!\!\text{Alg}} \longrightarrow \underline{\text{Sets}}
\]
associates to any Noetherian $\kk$-algebra $A$ the set
$$\MFFunctor{J}(A):=\{ (\mathcal H) \subset R_A\mid \mathcal H\text{ is a } J\text{-marked basis}\}$$
and to any morphism of $\kk$-algebras $\sigma: A \rightarrow {A'}$ the map
\[
\renewcommand{\arraystretch}{1.3}
\begin{array}{rccc}
\MFFunctor{J}(\sigma):& \MFFunctor{J}(A) &\longrightarrow& \MFFunctor{J}({A'})\\
&(\mathcal H) & \longmapsto& (\sigma(\mathcal H))\;.
\end{array}
\]
Note that the image $\sigma(\mathcal H)$ under this map is indeed again a
$\mathcal P_J$-marked basis, as we are applying the functor $-\otimes_A {A'}$
to the decomposition
$\left(R_A\right)_s =  (\mathcal H)_s \oplus \langle
  \mathcal {N}(J)_s \rangle_A$  for every degree $s$.
  
\begin{remark}\label{rem:lemmaMP}
Generalising \cite[Proposition 2.1]{LR2} to quasi-stable ideals, we obtain
\[
\{ (\mathcal H) \subset R_A\mid \mathcal H\text{ is a } \mathcal P_J\text{-marked basis}\}=\{ I \subset R_A \text{ ideal } \mid  R_A=I\oplus \mathcal \langle \mathcal {N}(J)\rangle_A\}.
\]
\end{remark}

The functor $\MFFunctor{J}$ is represented by the affine scheme $\MFScheme{J}$ that can be explicitly constructed by the following procedure. We consider the $\kk$-algebra $\kk[C]$, where $C$ denotes the finite set of variables
$\bigl\{C_{\alpha\eta}\mid x^\alpha \in \PP{J}, x^\eta
\in \NN(J), \deg(x^\eta)=\deg(x^\alpha)\bigr\}$, and
construct the $\mathcal P_J$-marked set $\mathscr H\subset R_{\kk[C]}$ consisting of the following marked polynomials
\begin{equation}\label{polymarkKC}
h_\alpha=x^\alpha-\sum_{x^{\eta}\in \NN(J)_{\vert\alpha\vert}}C_{\alpha\eta }x^\eta 
\end{equation}
with $x^{\alpha}\in\PP{J}$.
According to \eqref{eq:multipli}, we consider $\mathscr H^{(t)}$, for every integer $t$. 

Then, by the Noetherian and confluent reduction procedure given in Definition~\ref{def:riscrittura}, for every term $x^\alpha \in \PP{J}$ and every non-multiplicative variable $x_i$ of $x^\alpha$, like in~\eqref{eq:standard representation} we compute a polynomial $g_{\alpha,i}\in \langle \NN(J)_{\vert\alpha\vert+1}\rangle_A$ such that $x_i h_\alpha-g_{\alpha,i}\in \langle \mathscr H^{(t)} \rangle_A$, for some integer $t$. 

We denote by $\mathscr U$ the ideal generated in $\kk[C]$ by the $x$-coefficients of the polynomials~$g_{\alpha,i}$. Hence, we have $\MFScheme{J}=\mathrm{Spec} (\kk[C]/\mathscr U)$ (see \cite[Remark 6.3]{CMR2015} and \cite[Theorem~5.1]{Quot}), thanks to Theorem \ref{th:equivalent conditions}.

If $J$ is in particular a {\em saturated} quasi-stable ideal, then $x_0$ does not divide any term of $B_{J}$ and $R/J$ has positive Krull dimension. For every integer $t$, $J_{\geq t}$ is quasi-stable too, so that we can consider $\MFScheme{J_{\geq t}}$. 

Let $p(z)$ be the Hilbert polynomial of $R/J$ and $\HScheme{\mathbb P^n}{p(z)}$ be the Hilbert scheme that describes flat families of closed subschemes of $\mathbb P^n$ having Hilbert polynomial $p(z)$. Then, $\MFScheme{J_{\geq t}}$ embeds in $\HScheme{\mathbb P^n}{p(z)}$, for every integer $t$, like a locally closed subscheme (see \cite[Proposition 6.13]{BCRAffine}). This result can be refined in the following way.

If $\PP{J}$ does not contain any term divisible by $x_1$, we set $\rho_{J}:=1$. Otherwise, we set $\rho_{J}:=\max\{\deg(x^\alpha) \ \vert \ x^\alpha \in \PP{J} \text{ is divisible by } x_1\}=\mathrm{sat}\bigl(\frac{(J,x_0)}{(x_0)} \bigr)$.

\begin{proposition}\label{prop:sottoschema}
\cite[Corollary 6.11, Proposition 6.13(ii)]{BCRAffine}
With the above notation, 
\begin{enumerate}
\item for every $t\geq \rho_{J}-1$, $\MFScheme{J_{\geq t}} \cong \MFScheme{J_{\geq t+1}}$;
\item for every $t\geq \rho_{J}-1$, $\MFScheme{J_{\geq t}}$ is an open subscheme of $\HScheme{\mathbb P^n_K}{p(z)}$.
\end{enumerate}
\end{proposition}

%\section{Generalities on Cohen-Macaulay ideals}
%\label{sec:generalities}

 \section{Cohen-Macaulay conditions by marked bases}
\label{sec:CM}

Let $I\subset R=\mathbb K[x_0,\dots,x_n]$ be a homogeneous ideal such that the Krull-dimension $\dim(R/I)$ of $R/I$ is $d$, with $\kk$ any field. If we denote by $M$ the graded $R$-module $R/I$, the codimension (or height) of $I$ is $\mathrm{codim}(I)=\dim R -\dim M$.%, being $R$ a domain. 

For any (graded) $R$-module $M$ we only take $M$-regular sequences that are made of homogeneous polynomials.
All the maximal $M$-regular sequences have the same length, which is called the {\em depth of $M$} and denoted by $\mathrm{depth}(M)$. In general, the inequality $\mathrm{depth}(M)\leq \dim(M)$ holds.

\begin{definition}
A graded $R$-module $M$ is called a {\em Cohen-Macaulay (CM for short) module} if and only if $\mathrm{depth}(M)=\dim(M)$. If $M=R/I$, then the ideal $I$ is said CM if and only if $M$ is CM. 
Analogously, we will say that the closed projective scheme defined by $I$ is \emph{arithmetically Cohen-Macaulay} 
%(aCM for short) 
if and only if $I^{sat}$ is~CM (see \cite{Mi}). 

The {\em arithmetically Cohen-Macaulay locus} in a Hilbert scheme is the subset of points corresponding to arithmetically Cohen-Macaulay schemes.
\end{definition}

From now we assume $M:=R/I$. Note that if $R/I$ is Artinian then it is CM.

If $\ell$ is a linear non-zero divisor on $M$, then $M$ is CM if and only if $M/(\ell)M$ is~CM. If $\ell_0,\dots,\ell_{d-1}$ is a maximal $M$-regular sequence made of linear forms, then the module $M/(\ell_0,\dots,\ell_{d-1})M$ is called an {\em Artinian reduction of $M$} and, analogously, $I/(\ell_0,\dots,\ell_{d-1})I \simeq (I+(\ell_0,\dots,\ell_{d-1}))/(\ell_0,\dots,\ell_{d-1})$ is an Artinian reduction of $I$. Up to a linear change of variables, we can assume that $x_0,\dots,x_{d-1}$ is a maximal $M$-regular sequence.

Our first aim is to explore effective methods to check if a homogeneous ideal $I$ is CM exploiting the features of marked bases {\em only}. Hence, from now we assume that {\em $I\subset R$ is a homogeneous ideal generated by a $\mathcal P_J$-marked basis $\mathcal H=\{h_1,\dots,h_t\}$ and $d$ is the Krull dimension of $M=R/I$}. 

Recall that the variables of the polynomial ring $R=\mathbb K[x_0,\dots,x_n]$ are ordered as $x_0<x_1<\dots<x_n$ and $J\subset R$ is a quasi-stable ideal.

By the properties of quasi-stable ideals, the sequence $x_0,\dots,x_{d-1}$ is a {\em generic sequence on $R/J$} in the sense that $x_i$ is not a zero-divisor on $R/(J,x_0,\dots,x_{i-1})^{sat}$, for every $i\in \{0,\dots,d-1\}$. 
Then, the sequence $x_0,\dots,x_{d-1}$ is a $R/J$-regular sequence if and only if $R/J$ is CM (see \cite[Proposition 2.20]{Seiler2009II}). Hence, $R/J$ is CM if and only if $J$ is generated by terms in $\mathbb K[x_{d},\dots,x_n]$.

Generally, it can happen that $J$ is not CM even if $I$ is CM, as the following example shows (differently from what happens when $J$ is the initial ideal of $I$ with respect to the degrevlex order).

\begin{example}\label{ex:regsale} %\cite[Example 1]{BCR2017}
Let $I$ be the ideal $(x_2^2,x_1x_2+x_0^2)\subset \mathbb  K[x_0,x_1,x_2]$, with $\mathrm{char}(\mathbb K)=0$ and $x_0<x_1<x_2$. The ideal $I$ is CM and, for every term order, its initial ideal is $(x_2^2, x_1x_2,x_0^2x_2,x_0^4)$. If $\prec$ is the degrevlex term order, then $\mathrm{gin}(I)=(x_2^2, x_1x_2,x_1^3)$ is a CM ideal. If $\prec$ is the deglex term order, then $\mathrm{gin}(I)=(x_2^2,x_1 x_2,x_0^2x_2,x_1^4)$ is a quasi-stable and non-CM ideal, on which the image of $I$ by a generic change of variables has a marked basis. 
\end{example}

On the other hand, if $J$ is CM then $I$ is CM, as the following result shows. This has already been stated in \cite[Corollary~3.9]{BCRAffine} with a hint for its proof. Here we give a proof in terms of the properties of quasi-stable ideals only. 

\begin{theorem} \label{th:CM}
Let $I$ be an ideal generated by a $\mathcal P_J$-marked basis. If $J$ is CM then $I$ is CM.
\end{theorem}

\begin{proof}
Since $J$ is CM, $\dim (R/J)=\mathrm{depth} (R/J)$. Since the ideal $I$ is
generated by a $\mathcal P_J$-marked basis, $d=\dim (R/I)=\dim (R/J)$.  The free resolution of the quasi-stable ideal $J$ induced by its Pommaret basis is generally non minimal, but its length is the projective dimension of $R/J$ (i.e.\ minimization does not affect the length of the resolution) implying $\mathrm{pd}(R/J)=n+1-d$  \cite[Theorem 8.11]{Seiler2009II}.  By \cite[Corollary 6.8]{Quot}, $\mathrm{pd}(R/I)\leq \mathrm{pd}(R/J)=n+1-d$. A strict inequality would imply, by the Auslander-Buchsbaum formula, $\mathrm{depth}(M)>d$, which is not possible since $\mathrm{depth}(M)\leq \dim(R/I)=d$. Hence, $\mathrm{pd}(R/I)=n+1-d$ and we conclude that $I$ is CM, too.
\end{proof}

Recall that zero-dimensional projective schemes are always arithmetically Cohen-Macaulay. Thus, Hilbert schemes corresponding to constant Hilbert polynomials are made of arithmetically Cohen-Macaulay schemes. 

For Hilbert polynomials of positive degree we can now recover the result that the arithmetically Cohen-Macaulay locus in the corresponding Hilbert scheme is an open subset, in terms of marked schemes.

\begin{corollary}\cite[Remark 3.10]{BCRAffine}\label{cor:CM} 
The  arithmetically Cohen-Macaulay locus in a Hilbert scheme with a non-constant Hilbert polynomial coincides with the union of the open subschemes $\MFScheme{J}$, with $J$ CM quasi-stable ideal, and of their images by linear changes of variables.
\end{corollary}

\begin{proof}
Since the ideal $J$ is CM, then $\rho_J=1$, and hence $\MFScheme{J}\simeq \MFScheme{J_{\geq t}}$ for every integer $t\geq 0$. So, thanks to Proposition \ref{prop:sottoschema}, $\MFScheme{J}$ is an open subscheme of the corresponding Hilbert scheme, for every $J$ CM, and it is made of Cohen-Macaulay schemes, by Theorem \ref{th:CM}. On the other hand, if $K$ is a CM ideal defining a Cohen-Macaulay scheme in a certain Hilbert scheme, we can find a deterministic change of variables $g$ 
 (see \cite{HSS}) such that $g(K)$ has a quasi-stable CM ideal $J$ as initial ideal with respect to the degree reverse lexicographic order. So, up to a suitable change of variables the ideal $K$ belongs to $\MFFunctor{J}(\kk)$.
\end{proof}

The use of changes of coordinates in Corollary \ref{cor:CM} is unavoidable, because $\MFFunctor{J}(\kk)$ can contain CM ideals even if $J$ is not CM, as we have already highlighted in Example~\ref{ex:regsale}. 
Thus, the following question arises: if we consider $\MFScheme{J}$ with $J$ non-CM, how can we detect $I\in \MFFunctor{J}(\kk)$ such that $I^{\mathrm{sat}}$ is CM, using the features of marked bases only? 

In Section \ref{sec:truncation} we will give an answer to this question in the particular situation that $J=(J^{sat})_{\geq m}$ and consequently $I=(I^{sat})_{\geq m}$, for a suitable integer $m$ (see \cite[Corollary~3.7]{BCRAffine}). This is the situation that allows us to embed marked schemes in Hilbert schemes, in the further hypothesis that the Krull-dimension of $R/J$ is $d\geq 1$, as recalled in Section \ref{sec:marked functors}.

Finally, the following refinement of the result of Theorem~\ref{th:CM} gives us a suitable construction of Artinian reductions. 

\begin{proposition}\label{prop:Artinian reduction}
Let $I$ be an ideal generated by a $\mathcal P_J$-marked basis. If $x_0,\dots,x_{d-1}$ is a $R/J$-regular sequence, then 
\begin{itemize}
\item[(i)] $x_0,\dots,x_{d-1}$ is a $R/I$-regular sequence too, and 
\item[(ii)] the polynomials obtained from the $\mathcal P_J$-marked basis of~$I$ setting $x_0=\dots=x_{d-1}=0$ form a marked basis of the Artinian reduction of $I$ over the quotient $(J+(x_0,\dots,x_{d-1}))/(x_0,\dots,x_{d-1})$. 
\end{itemize}
\end{proposition}

\begin{proof}
If $x_0,\dots,x_{d-1}$ is a $R/J$-regular sequence, then $J$ is CM and $I$ is CM too by Theorem~\ref{th:CM}. Then, item (i) follows by applying recursively \cite[Theorem~3.5]{BCRAffine}.

Item (ii) follows from the fact that, being every hyperplane section of $I$ saturated because $I$ is Cohen-Macaulay, the differences of its Hilbert function coincide with the Hilbert function of the hyperplane sections. Then we conclude by Theorem~\ref{th:equivalent conditions}(ii).
\end{proof}

\section{Marked schemes over a truncated quasi-stable ideal}
\label{sec:truncation}

We here focus on the identification of Cohen-Macaulay ideals generated by a $\mathcal P_J$-marked schemes in the particular situation $J=(J^{sat})_{\geq m-1}$ and then $I=(I^{sat})_{\geq m-1}$, where $m\geq \rho_J$. 

First we recover a technical lemma which concerns hyperplane sections. Recall that if $I^{sat}$ has a $\mathcal P_{J^{sat}}$-marked basis then $(I^{sat})_{\geq t}$ has a $\mathcal P_{(J^{sat})_{\geq t}}$-marked basis, but the converse is not always true (see \cite[Example 3.8]{BCRAffine}). 

\begin{lemma}\label{lemma:lemma}\cite[Lemma 9.4]{BCF}
Let $J\subset R$ be a saturated quasi stable ideal such that $d=\dim(R/J)>0$, $J':=\bigl(\frac{(J,x_0)}{(x_0)}\bigr)^{sat}$ and $\rho$ be the satiety of $(J,x_0)/(x_0)\subset \mathbb K[x_1,\dots,x_n]$. For every $m\geq \rho$, if $I$ belongs to $\MFFunctor{J_{\geq m-1}}(\kk)$, then $\left(\frac{(I,x_0)}{(x_0)}\right)_{\geq m}$ belongs to $\MFFunctor{J'_{\geq m}}(\kk)$.
\end{lemma}

Recall that $x_0,\dots,x_{d-1}$ is a generic sequence for every {\em quasi-stable} ideal $H$ with $R/H$ of Krull-dimension $d$.

\begin{lemma}\label{lemma:generic sequence}
If $J=(J^{sat})_{\geq m-1}$ and $I=(I^{sat})_{\geq m-1}$, with $m$ like in Lemma~\ref{lemma:lemma}, then $x_0,\dots,x_{d-1}$ is a generic sequence on $R/I^{sat}$. 
\end{lemma}

\begin{proof}
By \cite[Theorem~3.5]{BCRAffine}, $x_0$ is generic on $R/I$. If $d=1$ we have finished. Otherwise, by Lemma \ref{lemma:lemma} and with the same notation, $\left(\frac{(I,x_0)}{(x_0)}\right)_{\geq m}$ belongs to $\MFFunctor{J'_{\geq m}}(\kk)$ and we can repeat the same argument on $\left(\frac{(I,x_0)}{(x_0)}\right)_{\geq m}$ obtaining that $x_0,x_1$ is a generic sequence on $R/(I^{sat})_{\geq m}$. Then we repeat the same argument until we obtain the thesis observing that $I^{sat}=((I^{sat})_{\geq m+d-2})^{sat}$.
\end{proof}

We have already observed that every  quasi-stable ideal $H$ with $R/H$ of Krull-dimension $d$ is CM if and only if $x_0,\dots,x_{d-1}$ is a $R/H$-regular sequence. This result can be extended to any ideal generated by a marked basis over the truncation of a quasi-stable ideal, under the same hypothesis of Lemma \ref{lemma:generic sequence}.

\begin{proposition} \label{prop:CM}
If $J=(J^{sat})_{\geq m-1}$ and $I=(I^{sat})_{\geq m-1}$ with $m$ like in Lemma~\ref{lemma:lemma}, then the ideal $I^{sat}$ is CM if and only if $x_0,\dots,x_{d-1}$ is a regular sequence on $R/I^{sat}$.
\end{proposition}

\begin{proof}
If $x_0,\dots,x_{d-1}$ is a regular sequence on $R/I^{sat}$ then $I^{sat}$ is CM by definition. Conversely, recall that   $x_0,\dots,x_{d-1}$ is a generic sequence on $R/I^{sat}$, by Lemma \ref{lemma:generic sequence}. Since $I^{sat}$ is CM by hypothesis, then its generic linear sections are saturated too and hence $x_0,\dots,x_{d-1}$ is a $R/I^{sat}$-regular sequence.
\end{proof}

In the same hypotheses of Proposition \ref{prop:CM}, we give the following computational strategy to check if $I^{sat}$ is~CM.

\begin{algo} \label{algo 1}
Let $J=(J^{sat})_{\geq m-1}$ be a quasi-stable ideal and $I=(I^{sat})_{\geq m-1}\in \MFFunctor{J_{\geq m}}(\kk)$ with $m \geq \mathrm{sat}\bigl(\frac{(J,x_0)}{(x_0)}\bigr)$. Let $d:=\dim(R/J)$. The following instructions allow to check if $I^{sat}$ is CM or not.
\begin{itemize}
\item[(1)] Compute $I^{sat}$ and set $k:=1$.
\item[(2)] Compute a marked basis of the ideal $N:=\bigl(\frac{(I,x_{k-1})}{(x_{k-1})} \bigr)_{\geq m-1}$ and then compute its saturation $N^{sat}$.
\item[(3)] If the first difference of the Hilbert function of $I^{sat}$ does not coincide with the Hilbert function of $N^{sat}$, then $I^{sat}$ is not CM and the procedure ends. Otherwise, if $d=1$ then $I^{sat}$ is~CM and the procedure ends, if $d>1$ then reset $d:=d-1$, $k:=k+1$, $I:=N$ (hence $I^{sat}=N^{sat}$), $m:=\max\left\{m, \mathrm{sat}\bigl(\frac{(J,x_0,\dots,x_{k-1})}{(x_0,\dots,x_{k-1})}\bigr)\right\}$, and go to step (2). 
\end{itemize}
\end{algo}

\begin{proof}
For what concerns item (1), we observe that the equality $I^{sat}=(I:x_0^{\infty})$ holds thanks to \cite[Theorem 3.5]{BCRAffine} and we will give a method to compute it by marked bases in next Theorem \ref{th:saturation}.

For what concerns item (2), thanks to Lemma \ref{lemma:lemma} the ideal $N:=\bigl(\frac{(I,x_{k-1})}{(x_{k-1})} \bigr)_{\geq m-1}$ belongs to $\MFFunctor{\frac{(J,x_0)}{(x_0)}_{\geq m-1}}(\kk)$ and we can compute a marked basis of $N$. Moreover, we can compute $N^{sat}$ thanks to  next Theorem \ref{th:saturation} because $N^{sat}=(N:x_{k}^{\infty})$ by \cite[Theorem~3.5]{BCRAffine} again.

For what concerns item (3), it is enough to observe that the check on the Hilbert functions is equivalent to check if $x_0,\dots,x_{d-1}$ is a $R/I^{sat}$-regular sequence, and hence that $I^{sat}$ is CM by Proposition \ref{prop:CM}. 
\end{proof}

The strategy of Algorithm \ref{algo 1} is pretty standard, except for the computational method that we now propose for the saturation of the ideals involved in the strategy. Indeed, by arguments analogous to those we use in Section \ref{sec:Gore}, we obtain the following description of $I^{sat}=(I:x_0^{\infty})$. 
%QUI (for the computation of the saturation of any ideal in terms of Gr\"obner bases, a very recent improvement is given in EDER 2023 https://arxiv.org/abs/2202.13387v2).

\begin{theorem} \label{th:saturation}
Let $J$ and $I$ be ideals in $R$ such that, for some $m\geq \rho$,  $J=(J^{sat})_{\geq m-1}$ and $I=(I^{sat})_{\geq m-1}$, and $I$ is generated by the $\mathcal P_J$-marked basis $\mathcal H$.

Let $\mathcal H_0=\{h_{\alpha_1},\dots,h_{\alpha_r}\}\subseteq \mathcal H$ be the set  made of the marked polynomials in $\mathcal H$ with head term divisible by $x_0$.
For every polynomial $h_{\alpha_i}\in \mathcal H_0$, let $h_{\alpha_i}=h'_{\alpha_i,k}+h''_{\alpha_i,k}$ be the decomposition of $h_{\alpha_i}$ such that the terms in $h'_{\alpha_i,k}$ are divisible by $x_0^k$ and the terms in $h''_{\alpha_i,k}$ are not divisible by~$x_0^k$. 

Let $\mathcal S$ be the set
\[
\mathcal S= \left\{\sum_{i=1}^r c_{\alpha_i,k} \frac{h'_{\alpha_i,k}}{x_0^k} : c_{\alpha_i,k}\in \kk, x_0^k \ \vert \ \mathrm{Ht}(h_{\alpha_i}), \sum_{i=1}^r c_{\alpha_i,k} h''_{\alpha_i,k} =0, \ k \in\{1,\dots,m-2\} \right\}
\]

Then, we have the following (graded) decomposition:
$$(I:x_0^\infty)= I \oplus \left\langle \mathcal S\right\rangle_{\mathbb K}.$$
%\todo[inline]{does it makes sense to mix up $R$-modules and a $\mathbb K[x_0]$-module?-C}
%\todo[inline]{otherwise maybe better: for every $t\geq 0$,$(I:x_0^\infty)_t= I_t \oplus \left\langle \mathcal S\cdot x_0^{t-m}\right\rangle_{\mathbb K}$, where the  second summand is a $\mathbb K$-vector space and $S\cdot x_0^{t-m}$ should be clear -C }
%\todo[inline]{should we consider the ideal generated by $\mathcal S$ as second summand, see Example after the proof?- C}
\end{theorem}
%Francesca: I think that we should argue in terms of $K$-vector spaces, clarifying that the coefficients $c_{\alpha_i,k}$ vary in all the filed $\kk$.

\begin{proof}
Since $(I:x_0^\infty)\supseteq I$, to obtain a first inclusion it is sufficient to prove that $(I:x_0^\infty)$ contains any polynomial  in $\mathcal S$. 

Consider $g\in \mathcal S$: $g=\sum_{i=1}^r c_{\alpha_i,k} \frac{h'_{\alpha_i,k}}{x_0^k}$, with $\sum_{i=1}^r c_{\alpha_i,k} h''_{\alpha_i,k} =0$. So, we can write
\[
x_0^kg=\sum_{i=1}^r c_{\alpha_i,k}h'_{\alpha_i,k}+\sum_{i=1}^r c_{\alpha_i,k} h''_{\alpha_i,k} =\sum_{i=1}^r  c_{\alpha_i,k}h_{\alpha_i}.
\]
This proves that $g$ belongs to $(I:x_0^k)$.

In order to prove the other inclusion, under the current hypotheses on $J$ and $I$, first we note that every polynomial in $\mathcal H_0$ has degree $m-1$ by \cite[Lemma 3.4]{BCRAffine}. 

Consider $f\in (I:x_0^\infty)$. By using the $\mathcal P_J$-marked basis $\mathcal H$, we obtain the writing $f=\sum_{h_\alpha \in\mathcal H} P_\alpha h_\alpha +\tilde f$, where the support of $\tilde f$ is contained in $\mathcal N(J)$.

If $\tilde f=0$, then $f$ belongs to $I$. Otherwise, we consider the smallest exponent $k$ such that $x_0^kf\in I$. Again by \cite[Lemma 3.4]{BCRAffine}, $x_0^kf$ has degree $m-1$. Furthermore, $x_0^k\tilde f$ belongs to $I$ too, hence we can rewrite it by the polynomials in $\mathcal H$.

Let $\tau$ be a term in $\supp(\tilde f)$ such that $x_0^k\tau \in J$. Hence, there is $x^\alpha \in \mathcal P_{J}$ such that $x_0^k\tau= x^\delta x^\alpha$ with $\max(x^\delta)\leq \min(x^\alpha)$. By \cite[Lemma 2.7 (iv)]{BCRAffine}, $\min(x^\alpha)=x_0$ and being $\deg(x^\alpha)=m-1=k+\deg(\tilde f)$, we have that $x_0^k$ divides $x^\alpha$. Hence in the rewriting procedure on $x_0^k\tilde f$ we use only polynomials in $\mathcal H_0$ whose head term is divided by $x_0^k$, and every new term that is introduced by a rewriting step belongs to the sous-escalier of $J$ and hence it is not rewritable.

Then  we can write:
\[x_0^k \tilde f = \sum_{i=1}^r c_{\alpha_i,k} h_{\alpha_i} = \sum_{i=1}^r c_{\alpha_i,k} h'_{\alpha_i,k} + \sum_{i=1}^r c_{\alpha_i,k} h''_{\alpha_i,k}.\]
This is possible if and only if $\sum c_{\alpha_i,k} h''_{\alpha_i,k}=0$, where the coefficients $c_{\alpha_i,k}$ belong to~$\mathbb K$.
\end{proof}

\begin{example} 
In the ring $\mathbb K[x_0,x_1,x_2,x_3]$ with $x_0<\dots<x_3$, consider the saturated quasi-stable but not CM ideal $J=(x_3^2,x_2x_3,x_1^2x_3,x_2^4)$ with $\rho=3$. The Krull dimension of the quotient over $J$ is~$2$. We take $m=4$ and the truncation $J_{\geq 3}=(x_3^3, x_2x_3^2, x_2^2x_3,$ $x_1x_3^2, x_1x_2x_3, x_0x_3^2,$ $x_1^2x_3, x_0x_2x_3, x_2^4)$ and the ideal $I$ generated by the following $\mathcal P_{J_{\geq 3}}$-marked basis

$\{ x_3^3,  \ x_2x_3^2,  \ x_2^2x_3 + x_2^3 + 2x_1x_2^2 + x_1^2x_2, \ x_1x_3^2,  \ x_1x_2x_3 + x_1x_2^2 + 2x_1^2x_2 + x_1^3,$ 

$\ x_0x_3^2, \ x_1^2x_3 - x_2^3 - 4x_1x_2^2 - 5x_1^2x_2 - 2x_1^3,
\ x_0x_2x_3 + x_0x_2^2 + 2x_0x_1x_2 + x_0x_1^2,$

$\ x_2^4 + 4x_1x_2^3 + 6x_1^2x_2^2 + 4x_1^3x_2 + x_1^4
\}.$

\noindent By \cite[Lemma 9.4]{BCF}, in $\mathbb K[x_1,x_2,x_3]$, we compute the marked basis of $N:=\Bigl(\frac{(I,x_0)}{(x_0)}\Bigr)_{\geq 3}$ on $\Bigl(\frac{(J,x_0)}{(x_0)}\Bigr)_{\geq 3} \ = \
\{x_3^3, x_2x_3^2, x_2^2x_3, x_1x_3^2, x_1x_2x_3,  x_1^2x_3,  x_2^4\}$, obtaining

$\mathcal H=\{ h_1=x_3^3,  \ h_2=x_2x_3^2,  \ h_3=x_2^2x_3 + x_2^3 + 2x_1x_2^2 + x_1^2x_2, \ h_4=x_1x_3^2,$  

$\ h_5=x_1x_2x_3 + x_1x_2^2 + 2x_1^2x_2 + x_1^3,\ h_6=x_1^2x_3 - x_2^3 - 4x_1x_2^2 - 5x_1^2x_2 - 2x_1^3,$

$\ h_7=x_2^4 + 4x_1x_2^3 + 6x_1^2x_2^2 + 4x_1^3x_2 + x_1^4
\}$

\noindent and applying Theorem \ref{th:saturation} with $x_1$ in place of $x_0$ 
$$N^{sat}=N\oplus \langle x_3^2, x_2x_3+x_2^2+2x_1x_2+x_1^2\rangle_{\kk}= (x_3^2, x_2x_3+x_2^2+2x_1x_2+x_1^2).$$
The Hilbert function of $I^{sat}$ is $1 \ 4t$ and its first derivative is $1 \ 3 \ 4 \ 4 \dots$. Since the Hilbert function of $N^{sat}$ is $1 \ 3 \ 4 \ 4 \dots$ too, we can conclude that $I^{sat}$ is CM. 

In order to give some more details of the computation of $N^{sat}$, consider $\mathcal H_1=\{ h_4=x_1x_3^2, \ h_5=x_1x_2x_3 + x_1x_2^2 + 2x_1^2x_2 + x_1^3,\ h_6=x_1^2x_3 - x_2^3 - 4x_1x_2^2 - 5x_1^2x_2 - 2x_1^3\}$. 

For $k=1$ the condition in Theorem \ref{th:saturation} is:

$c_{4,1}x_3^2+c_{5,1}(x_2x_3+x_2^2+2x_1x_2+x_1^2)+c_{6,1}(x_1^2x_3 - 4x_1x_2^2 - 5x_1^2x_2 - 2x_1^3)$

\noindent such that $c_{6,1}x_2^3=0$, which implies $c_{6,1}=0$ and $c_{4,1},c_{5,1}\in \kk$.

For $k=2$ the condition in Theorem \ref{th:saturation} is:

$c_{6,2}(x_3 - 5x_2 - 2x_1)$ 

\noindent such that $c_{6,2}(-x_2^3 -  4x_1x_2^2)=0$, which implies $c_{6,2}=0$.
\end{example}

\begin{example}
This example is inspired from \cite[Example 9.10]{BCF}. Consider the saturated quasi-stable but not CM ideal 
$J=J^{sat}=(x_2x_4,x_4^2,x_1x_3x_4,x_2x_3^2,x_3^3,x_3^2x_4)$ in $\mathbb K[x_0,\dots,x_4]$, with $m=\rho=3$ and $x_0<\dots<x_4$. Observe that $J_{\geq 2}$ coincides with $J$. The saturated ideal $I^{sat}=I=(x_2x_4-x_3^2-x_3x_4,x_4^2,x_1x_3^2+x_1x_3x_4,x_2x_3^2,x_3^3,
x_3^2x_4)$ coincides with $I_{\geq 2}$ and is generated by a $\mathcal P_{J_{\geq 2}}$-marked basis. 
By \cite[Lemma 9.4]{BCF}, in $\mathbb K[x_1,\dots,x_4]$ we can compute the basis of $\Bigl(\frac{(I,x_0)}{(x_0)}\Bigr)_{\geq 3}$ that is marked on the Pommaret basis of $\Bigl(\frac{(J,x_0)}{(x_0)}\Bigr)_{\geq 3}=( x_4^3, x_3x_4^2, x_2x_4^2, x_1x_4^2, x_2x_3x_4, x_2^2x_4,$ $x_3^2x_4, x_2x_3^2, x_3^3, x_1x_2x_4,$ $x_0x_2x_4, x_0x_4^2, x_1x_3x_4)$, \hskip 1mm and obtain
$$N:=(x_4^3,x_3x_4^2,x_2x_4^2,x_1x_4^2, x_2x_3x_4, x_2^2x_4,x_3^2x_4,
x_2x_3^2,x_3^3,x_1x_2x_4-x_1x_3^2, x_1x_3^2+x_1x_3x_4)$$
and applying Theorem \ref{th:saturation}
$$N^{sat}=(x_4^3,x_4^2,
x_2x_3x_4,x_3^2x_4,x_2x_3^2,x_3^3,x_2x_4-x_3^2,x_3^2+x_3x_4).$$
The Hilbert function of $I^{sat}$ is $1 \ 5 \ t^2+4t+1$ and its first difference is $1 \ 4 \ 8 \ 2t+3$, but the Hilbert function of $N^{sat}$ is different, being $1 \ 4 \ 2t+3$. Hence, $I^{sat}$ is not a CM ideal.
\end{example}

\section{Gorenstein conditions by marked bases}
\label{sec:Gore}
%\todo[inline]{We need to remember to add the variable $x_0$ and to modify everything. F (Done---But could be useful to keep this note here in a tex comment. M)}

In a polynomial ring over a field $\kk$, the study of Gorenstein homogeneous ideals can be reduced to the study of Artinian homogeneous ideals, like for CM ideals.

Indeed, a Cohen-Macaulay ideal $I$ is Gorenstein if and only if its Artinian reduction is Gorenstein or, equivalently, the socle of its Artinian reduction has dimension~$1$ as a $\mathbb K$-vector space or, equivalently, the last module of its minimal free resolution has rank $1$ (see \cite[Proposition~21.5 and Corollary~21.16]{Ei}). 
%This last equivalence holds because minimal free resolutions are preserved by general hyperplane sections (for example, see \cite[Theorem 1.3.6]{Mi}). 
If $I\subseteq R$ is a Gorenstein ideal, we say that $R/I$ is a Gorenstein ring. 
It is noteworthy that the Hilbert function of every Artinian graded Gorenstein $\mathbb K$-algebra is symmetric.

A closed projective scheme defined by a homogeneous polynomial ideal $I$ is \emph{arithmetically Gorenstein} if and only, if $I^{sat}$ is Gorenstein. 

The {\em arithemtically Gorenstein locus} in a Hilbert scheme is the subset of points corresponding to arithmetically Gorenstein schemes.

Thanks to Proposition \ref{prop:Artinian reduction}, if $J$ is a CM quasi-stable ideal with $d$ as Krull dimension of $R/J$ and $I$ is an ideal with a $\mathcal P_J$-marked basis, then the quotient $(I+(x_0,\dots,x_{d-1}))/(x_0,\dots,x_{d-1})$ is an Artinian reduction of $I$ with marked basis over the Artinian quasi-stable ideal $(J+(x_0,\dots,x_{d-1}))/(x_0,\dots,x_{d-1})$. Hence, $I$ is Gorenstein if and only if $(I+(x_0,\dots,x_{d-1}))/(x_0,\dots,x_{d-1})$ is Gorenstein.

\begin{remark}\label{rem:Gore}
Recall that a monomial ideal is Gorenstein if and only if it is a complete intersection (see \cite{Beintema}, for example). The ideal $I$ in Example \ref{ex:regsale} is Gorenstein and is generated by a $\mathcal P_J$-marked basis, where $J$ is a non-Gorenstein quasi-stable ideal.
On the other hand, in Example \ref{ex:example one} we will  find a Gorenstein quasi-stable ideal $J$ and a $\mathcal P_J$-marked basis $\mathcal H$ generating an ideal which is not Gorenstein.
\end{remark}

We now aim to describe the non-trivial elements of the socle of an Artinian ideal generated by a $\mathcal P_J$-marked basis. To this end, we adapt a method for socle computation due to Seiler~\cite[Theorem 5.4]{Seiler2012} (slightly corrected in~\cite[Satz 63]{Matthias}) which is valid for Artinian ideals generated by Pommaret (Gr\"obner) bases with respect to the degree reverse lexicographic term order. We use the following notation.%, following~\cite[Sec. 4.2]{Matthias}.

%\todo[inline]{We need to introduce the right notation. F \& M}

Let $I$ be an Artinian homogeneous ideal generated by a $\mathcal P_J$-marked basis $\mathcal H=\{h_1,\dots,h_t\}$ and let $\mathcal H_0=\{h_{\alpha_1},\dots,h_{\alpha_r}\}$ be the subset of $\mathcal H$ made of the polynomials with head term divisible by $x_0$. For every polynomial $h\in \mathcal H_0$, let $h=h'+h''$ be the decomposition of $h$ such that $h'$ is divisible by $x_0$ and $h''$ is linear combination of terms that are not divisible by $x_0$.

For each $k\in\{1,\ldots,n\}$, we associate to $\mathcal{H}$ the square matrix $A_k(\mathcal{H})\in\kk^{r\times r}$ whose entry in row $i$ and column $j$ is the constant term of the coefficient polynomial $P_{h_{\alpha_i}}$ in the {\em standard representation} $x_k h_{\alpha_j}=\sum_{h\in\mathcal{H}}P_h h$ (see \eqref{eq:standard representation}).

%\todo[inline]{"standard representation" or "decomposition"? Need to choose one terminology and stick to it throughout the article.---M.}
%\todo[inline]{according to the papers we quote, we should choose "standard representation".---F.}

\begin{theorem}\label{th:socle}
%\todo[inline]{The matrix $A_k$ is relative to the polynomials $h$ or $h'$ only? (maybe it is the same) 
%the set appearing as second summand is a $\mathbb K$-vector space (we are now in the setting of Artinian ideals -C}
Let $I$ be an Artinian homogeneous ideal generated by the $\mathcal P_J$-marked basis $\mathcal{H}$, with associated matrices $A_k:=A_k(\mathcal{H})$ for $k=1,\ldots,n$. Let ${\mathfrak m}$ be the irrelevant maximal ideal in $\mathbb K[x_0,\dots,x_n]$. Then the ideal quotient $(I:\mathfrak{m})$ is a direct sum of $\kk$-vector spaces as follows:
$$(I:{\mathfrak m})= I \oplus \langle\Bigl\{\sum_{i=1}^r c_i \frac{h'_{\alpha_i}}{x_0} : c_i\in \kk, \sum_{i=1}^r c_i h''_{\alpha_i} =0 \text{ and } A_k {\bf c}=\mathbf{0}, \forall \ k\in \{1,\dots,n\} \Bigr\}\rangle_{\kk}.$$
\end{theorem}

\begin{proof}
If $f$ is a polynomial in $(I:{\mathfrak m})$, then in particular $x_0 f$ belongs to $I$ and we can represent $x_0 f$ by the rewriting procedure with $\mathcal H$:
$$x_0f= \sum_{h\in \mathcal H} P_h h.$$
Recalling that the terms of $x_0 f$ are all divisible by $x_0$, observe that $x_0 \tau$ belongs to $J$ if and only if there exists $h\in \mathcal H$ such that $x_0\tau=x_0^\ell x^\delta {\mathrm Ht}(h)$, where $x^\delta$ is not divisible by $x_0$ and $\max(x^\delta)\leq \min(\mathrm{Ht}(h))$. Two cases can now occur:
\begin{itemize}
\item[(a)] $\ell >0$, and hence at least one term in $P_h$ is divisible by $x_0$
\item[(b)] $\ell =0$, and hence $\mathrm{Ht}(h)$ is divisible by $x_0$; thus, $x^\delta=1$ and $x_0 \tau =\Ht(h)$.
\end{itemize}
In case (a) every new term that is introduced by the rewriting procedure is divisible by $x_0$. In case (b) every new term that is introduced by the rewriting procedure belongs to the sous-escalier of $J$ and so it is not rewritable.

Hence we can write:
$$x_0 f = \sum_{x_0 \vert \bar P_h} \bar P_h h + \sum_{\bar P_{\alpha_i}=c_i \in \mathbb K} \bar P_{\alpha_i} h_{\alpha_i}= \sum_{x_0 \vert \bar P_h} \bar P_h h + \sum_{\bar P_{\alpha_i}=c_i \in \mathbb K} \bar P_{\alpha_i} h'_{\alpha_i} + \sum_{\bar P_{\alpha_i}=c_i \in \mathbb K} \bar P_{\alpha_i} h''_{\alpha_i},$$
and this is possible if and only if $ \sum c_i h''_{\alpha_i}=0$.

As a consequence, we obtain
$$f=\sum_{x_0 \vert \bar P_h} \frac{\bar P_h}{x_0} h + \sum c_i \frac{h'_{\alpha_i}}{x_0}.$$
and for every $k>0$ we now have:
$$x_k f= \sum_{x_0 \vert \bar P_h} \frac{\bar P_h}{x_0} x_k h + x_k \sum c_i \frac{h'_{\alpha_i}}{x_0}.$$
Hence, $x_k f$ belongs to $I$ if and only if  
$x_k \sum c_i \frac{h'_{\alpha_i}}{x_0}$ belongs to $I$. This is equivalent to having the standard representation
$$x_k \sum c_i \frac{h'_{\alpha_i}}{x_0} = \sum_{h\in \mathcal H} Q_h h$$
and hence
$$x_k \sum c_i h'_{\alpha_i} = \sum_{h\in \mathcal H} x_0 Q_h h$$
which is a standard representation too, as the variable $x_0$ is multiplicative for every term (see \eqref{eq:standard representation}). This implies that the components of the standard representation of $x_k \sum c_i h'_{\alpha_i}$ are free of constant terms, for which we use the following notation:
\begin{equation}\label{eq:DIV}
\Sr(x_k \sum c_i h'_{\alpha_i})_0
=\mathbf{0},
\end{equation}
where $\Sr(x_k \sum c_i h'_{\alpha_i})$ is the vector of the coefficient polynomials of the standard representation of $x_k \sum c_i h'_{\alpha_i}$, as already set in Section \ref{sec:preliminaries}.

Recall that $\sum c_i h''_{\alpha_i}=0$. Using the additivity of standard representations and~\eqref{eq:DIV} we now deduce
\begin{equation}\label{eq:UseSrAdditivity}
\begin{split}
    \mathbf{0}=\Sr(x_k \sum c_i h'_{\alpha_i})_0+\Sr(0)_0=\Sr(x_k \sum c_i h'_{\alpha_i}+x_k \sum c_i h''_{\alpha_i})_0 \\
    =\Sr(x_k \sum c_i h_{\alpha_i})_0.
    \end{split}
\end{equation}
Observe that also the opposite implication holds true, because every term in $h'_{\alpha_i}$ is divisible by $x_0$. Moreover, \eqref{eq:UseSrAdditivity} is equivalent to the conditions $ A_k {\bf c}=0$, for every $k\in \{1,\dots,n\}$, because the possible non-null coefficients of the polynomials $h\in \mathcal H\setminus \mathcal H_0$ must be divisible by $x_0$ by construction.
\end{proof}

Let $\Sigma_{\mathcal H}$ denote the system of homogeneous linear equations given by the conditions describing the socle $(I:{\mathfrak m})/I$ of $I$ in Theorem \ref{th:socle}, in the $r$ variables $c_1,\dots,c_r$. Then, we obtain the following.

\begin{corollary}\label{cor:condition Gore}
Let $J$ be an Artinian monomial ideal.
\begin{itemize} 
\item[(i)] An ideal $I$ generated by a $P_J$-marked basis $\mathcal H$ is Gorenstein if and only if the associated matrix of coefficients of $\Sigma_{\mathcal H}$ has rank $r-1$.
\item[(ii)] The arithmetically Gorenstein locus in $\MFScheme{J}$ is an open subset ${\mathbf G}_J$ of~$\MFScheme{J}$.
\end{itemize}
\end{corollary}

\begin{proof}
Recall that the ideal $I$ is Gorenstein if and only if the $\mathbb K$-dimension of the socle of $I$ is equal to $1$, that is the vector space of solutions of the linear system $\Sigma$ has dimension $1$, thanks to  Theorem \ref{th:socle}. Being $r$ the number of variables in $\Sigma$, we obtain item (i).

Now, consider the $\mathcal P_J$-marked basis $\mathscr H\subseteq R_{\mathbb K[C]}$ as described in Section \ref{sec:marked functors}, modulo the ideal $\mathscr U$ defining the marked scheme $\MFScheme{J}$. By item (i),  Gorenstein ideals in $\MFScheme{J}$ are obtained if and only if at least a minor of order $r-1$ of the associated matrix of coefficients of the linear system $\Sigma_{\mathscr H}$ does not vanish. We conclude recalling that the socle of a proper homogeneous ideal $H$ is not null, because it contains the part of degree $s-1$ of the quotient $R/H$, when $s$ is its regularity.
\end{proof}

\begin{remark}\label{rem: Gore comp cost}
Since the number $r$ of variables of a linear system $\Sigma_{\mathcal H}$ is bounded from above by the cardinality of the sous-escalier $\mathcal N(J)$, the linear system $\Sigma$ has a number of equations of order $\mathcal O(\vert\mathcal N(J)\vert \cdot n)$ in $r\leq \vert\mathcal N(J)\vert$ variables.
\end{remark}

\begin{corollary}\label{cor:open Gore}
The arithmetically Gorenstein locus in a Hilbert scheme with a non-constant Hilbert polynomial is an open subset. 
\end{corollary}

\begin{proof}
Recall that every homogeneous ideal can be transformed into an ideal with a marked basis over a quasi-stable ideal by a certain linear change of variables, like suggested in \cite{HSS} and already done in the proof of Corollary \ref{cor:CM}. Hence, let $I$ be the homogeneous ideal generated by the marked basis $\mathscr H\subset R_{K[C]}$ over a Cohen-Macaulay quasi stable ideal $J$, modulo the ideal $\mathscr U$ defining the marked scheme $\MFScheme{J}$ (see Section \ref{sec:marked functors}). 
Being the Hilbert polynomial of $J$ non-constant by hypothesis, then $\rho_J=1$ and $\MFScheme{J}$ is an open subscheme in the corresponding Hilbert scheme, thanks to Proposition~\ref{prop:sottoschema}. 

In this particular situation, setting to zero the variables $x_0,\dots,x_{d-1}$ in the polynomials of the marked basis $\mathscr H$, we obtain the marked basis $\mathscr H'$ of an Artinian reduction $I'$ of $I$ over the Artinian reduction $J'$ of $J$. Recalling that an ideal is Gorenstein if and only if its Artinian reduction is Gorenstein, we now consider on the polynomials of $\mathscr H$ the conditions on the common coefficients of the polynomials of $\mathscr H'$ that define the open subset ${\mathbf G}_{J'}$ of $\MFScheme{J'}$ of Corollary \ref{cor:condition Gore}, obtaining an open subset ${\mathbf G}_{J}$ of $\MFScheme{J}$.

Hence, the arithmetically Gorenstein locus in Hilbert schemes with non-constant Hilbert polynomials coincides with the union of the open subsets ${\mathbf G}_{J}$ of $\MFScheme{J}$, with $J$ CM quasi-stable ideal, and of their images by linear changes of variables. 
\end{proof}

\begin{example}\label{ex:secGor}
Let us consider $J=(x^2,y^2)\subset \mathbb K[x,y]$ with $x<y$. The Pommaret basis of $J$ is $\mathcal P_{J}=\{x^2,y^2,x^2y\}$. For every constant value of the parameters $d_{1,1}, d_{2,2}$ the following polynomials form a $\mathcal P_{J}$-marked basis:
$$h_1=x^2+d_{1,1}xy, \quad h_2=y^2+d_{2,1}xy, \quad h_3=x^2y.$$
We have $\mathcal H_0=\{h_1,h_3\}$ and $h''_1=h''_3=0$. By the rewriting procedure we find
$yh_1=(1-d_{11}d_{2,1})h_3 + d_{1,1}xh_2$ and 
$yh_3=x^2h_2-d_{2,1}xh_3$,
so the matrix $A_1$ is 
$$\left(\begin{array}{cc}
0&0\\
1-d_{1,1}d_{2,1} & 0 \end{array}\right).$$
The system $\Sigma$ in the two variables $c_1,c_2$  is only made by the equation $(1-d_{1,1}d_{2,1})c_1=0$ and its associated matrix has rank $1$ if and only if $1-d_{1,1}d_{2,1}\not=0$. Under this condition, for every ideal $I=(h_1,h_2,h_3)$, the socle $(0:m)\subset R/I$ is generated by $[xy]\in R/I$. 
For example, the ideal $I=(x^2+xy,y^2-xy,x^2y)$ is Gorenstein.
\end{example}

\begin{example}\label{ex:Gore 2}
The quasi-stable ideal $J=(x_2^2,x_1^2,x_0^2)$ is Artinian in $\mathbb K[x_0,x_1,x_2]$ ($x_0<x_1<x_2$) and its Pommaret basis is $\mathcal P_{J}=\{x_2^2,x_1^2,x_0^2,x_0^2x_1,x_0^2x_2,x_1^2x_2,x_0^2x_1x_2\}$. The following polynomials form the $\mathcal P_J$-marked set $\mathcal H$:
\begin{equation}\label{eq:exNoGore2MS}
\begin{gathered}
h_1={x_{{0}}}^{2}+d_{{1,1}}x_{{0}}x_{{1}}+d_{{1,2}}x_{{0}}x_{{2}}+d_{{1,3}}x_{{1}}x_{{2}},\\
h_2={x_{{1}}}^{2}+d_{{2,1}}x_{{0}}x_{{1}}+d_{{2,2}}x_{{0}}x_{{2}}+d_{{2,3}}x_{{1}}x_{{2}},\\
h_3={x_{{2}}}^{2}+d_{{3,1}}x_{{0}}x_{{1}}+d_{{3,2}}x_{{0}}x_{{2}}+d_{{3,3}}x_{{1}}x_{{2}},\quad
h_4={x_{{0}}}^{2}x_{{1}}+d_{{4,1}}x_{{0}}x_{{1}}x_{{2}},\\
h_5={x_{{0}}}^{2}x_{{2}}+d_{{5,1}}x_{{0}}x_{{1}}x_{{2}},\quad
h_6={x_{{1}}}^{2}x_{{2}}+d_{{6,1}}x_{{0}}x_{{1}}x_{{2}},\quad
h_7={x_{{0}}}^{2}x_{{1}}x_{{2}}.
\end{gathered}
\end{equation}
The marked set $\mathcal H$ is a $\mathcal P_J$-marked basis if and only if the coefficients $d_{i,j}$ satisfy the equations listed in \eqref{eq:exGore2MF}.

We have $\mathcal H_0=\{h_1,h_4,h_5,h_7\}$ and $h_1''=d_{1,3}x_1x_2$, while $h''_4=h''_5=h''_7=0$. Hence,  we need to set $c_1(d_{1,3}x_1x_2)=0$ in order to compute the socle, so that $d_{1,3} c_1=0$ is one of the equations of the system $\Sigma_{\mathcal H}$. There are $6$ more non-null equations in the system $\Sigma_{\mathcal H}$, obtained from the matrices $A_0$, $A_1$, $A_2$ as in Theorem~\ref{th:socle}. The complete coefficient matrix of $\Sigma_{\mathcal H}$ can be seen at \eqref{eq:exNoGore2Matrix}. 
The complete list of equations describing the non-Gorenstein locus in the marked scheme defined by $J$ can be found at \eqref{eq:exGore2NoGore}. They are the  minors of order 3 of the matrix of the system $\Sigma_{\mathcal H}$. Indeed, these equations describe the loci were the socle of the ideal $(\mathcal H)$ has not dimension~$1$, as stated in Corollary \ref{cor:condition Gore}.
\end{example}

\section{Complete intersection conditions by marked bases}
\label{sec: CI conditions}

Referring to \cite{Kunz}, recall that a proper ideal $I$ in a Noetherian ring is called a {\em complete intersection} if the length of the shortest system of minimal generators of $I$ is equal to the height (or codimension) of $I$. A proper ideal $I$ that is generated by a regular sequence in a Noetherian ring is a complete intersection and the converse holds if the ring is Cohen–Macaulay, like a polynomial ring over a field.

A closed projective scheme defined by a homogeneous polynomial ideal $I$ is a \emph{(strict or global) complete intersection} if and only if $I^{sat}$ is a complete intersection (see \cite[Exercise 8.4, chapter II]{Hart-book}). 

The {\em strict complete intersection locus} in a Hilbert scheme is the subset of points corresponding to strict complete intersection schemes.

In this section we describe and use a method to identify Artinian complete intersection ideals among the ideals generated by marked bases using minimization in terms of linear algebra only. 

The strategy that we propose here is inspired by the minimization method of Gr\"obner bases that has been described in \cite[Section 3]{C1999} based on an interpretation of \cite[Lemma 13.1]{MMM} in terms of linear algebra only. The setting is that of zero-dimensional schemes which well fits with our problem. Indeed, we can consider Artinian ideals only, because complete intersections are preserved by general linear sections, as well.

Given an Artinian monomial ideal $J$, the main tool is the notion of border $\partial\mathcal O$ of the order ideal $\mathcal N(J)$, from which the concept of $\partial\mathcal O$-marked basis arises (see \cite{MMM,Mourrain1999} where a $\partial\mathcal O$-marked basis is called a {\em border basis}). In~\cite{BC2022} a comparison between a $\partial\mathcal O$-marked basis and a $\mathcal P_J$-marked basis is described. Here we develop our strategy referring to \cite{BC2022} for definitions and features related to $\partial\mathcal O$-marked bases and their relations with $\mathcal P_J$-marked bases, although in \cite{BC2022} both $\mathcal P_J$-marked bases and $\partial\mathcal O$-marked bases are not necessarily homogeneous.

Here we deal with homogeneous $\mathcal P_J$-marked bases, implying that also the corresponding $\partial\mathcal O$-marked bases must be homogeneous, thanks to \cite[Theorem~1]{BC2022}.

\begin{definition}\label{def:border}
Let $J$ be an Artinian ideal. The {\em border} of $J$ is 
$$\partial \mathcal O:=\{x_ix^\tau : x^\tau \in  \mathcal N(J) \text{ and }  i\in \{0,\dots,n\} \}\cap J.$$
\end{definition}

If $\mathcal H:=\{h_\alpha\}_{\alpha}$ is a $\mathcal P_J$-marked basis generating an ideal $I$, the following set of homogeneous marked polynomials
$$\mathcal B:=\{b_\tau:=x^\tau - \mathrm{Nf}_I(x^\tau) : x^\tau \in \partial \mathcal O \}$$
is the {\em (homogeneous) $\partial \mathcal O$-marked basis of $I$} (see \cite[Theorem~1]{BC2022}), being $\mathrm{Nf}_I(x^\tau)$ the normal form of $x^\tau$ by $I$ as defined in \eqref{eq:standard representation}, and $x^\tau$ the head term of the polynomial~$b_\tau$.
Indeed, the additional condition of homogeneity that we consider here does not forbid the application of \cite[Theorem~1]{BC2022}. 

Analogously to \eqref{eq:multipli}, for every $\partial\mathcal O$-marked basis $\mathcal B$ and for every integer $t$ we set
\begin{equation}\label{eq:multiples border}
\mathcal B^{(t)}:=\{x^\delta b_\tau : b_\tau \in \mathcal B, x^\delta=1 \text{ or } \max(x^\delta)\leq \min(x^\tau), \deg(x^\delta x^\tau)=t\}.
\end{equation}
If $x^\delta b_\tau$ belongs to $\mathcal B^{(t)}$ we say that $x^\delta x^\tau$ is its head term.
Observe that $\mathcal H^{(t)}$ is contained in $\mathcal B^{(t)}$ for every integer $t$. We highlight that two polynomials belonging to $\mathcal B^{(t)}$ can have the same head term, differently from $\mathcal H^{(t)}$.

\begin{definition}\label{def:riscrittura border}
Given a $\partial\mathcal O$-marked basis $\mathcal B=\{b_\tau\}_{\tau \in \partial\mathcal O}$, for every integer $t$ we denote by $\rid{\mathcal B^{(t)}}$ the transitive closure of the relation $f \rid{\mathcal B^{(t)}} f-\lambda x^\delta b_\tau$, where $f$ is a polynomial, $x^\delta x^\tau$ is a term that appears in $f$ with coefficient $\lambda$ and $x^\delta b_{\tau} \in \mathcal B^{(t)}$.
We will write $f \crid{\mathcal B^{(t)}} g$ if $f \rid{\mathcal B^{(t)}} g$ and $g \in \langle \mathcal N(J) \rangle_A$. 
\end{definition}

\begin{lemma}\label{lemma: multiples border}
If $x^\tau$ is a term that belongs to $\partial\mathcal O\setminus \mathcal P_J$, then it is a multiple of another term in the border by a multiplicative variable.
\end{lemma}

\begin{proof}
Let $x^\tau=x_u m$, with $m\in \mathcal N(J)$, and let $x_k:=\min(x^\tau)$. Then $x^\tau/x_k$ belongs to $J$, otherwise $x^\tau\in \mathcal P_J$. So $x_u\not=x_k$ and $x_u \frac{m}{x_k}$ belongs to $\partial\mathcal O$, because $\frac{m}{x_k}$ belongs to $\mathcal N(J)$.
\end{proof}

\begin{lemma}\label{lemma:lex}
Let $x^\delta m$ be a term that belongs to $J\setminus \partial\mathcal O$ with $m\in \mathcal N(J)$. Then, for every term $x^\gamma \in \partial\mathcal O$ and term $x^\eta$ such that $x^\delta m=x^\eta x^\gamma$ and $\max(x^\eta)\leq \min(x^\gamma)$, $x^\eta <_{\mathrm{lex}}~x^\delta$.
\end{lemma}

\begin{proof} 
With the same notation of the statement, if $x^\gamma$ belongs to $\mathcal P_J$, then we refer to \cite[Lemma 3.4(vi)]{Quot}.

Otherwise, there exist $m'\in \mathcal N(J)$ and a variable   $x_\ell$ such that $x^\gamma=x_\ell m'$, where $x_\ell>\min(x^\gamma)$ because $x^\gamma$ does not belong to $\mathcal P_J$. By Lemma \ref{lemma: multiples border}, $x^\gamma=x_k x^\sigma$, where $x^\sigma \in \partial\mathcal O$ and $x_k:=\min(x^\gamma)$. If $x^\sigma\in \mathcal P_J$ then $x^\eta x_k <_{\mathrm{lex}} x^\delta$ thanks to \cite[Lemma~3.4(vi)]{Quot} and hence $x^\eta <_{\mathrm{lex}}~x^\delta$, because $x_k=\min(x^\gamma)$. 
Otherwise we repeat the same argument on $x^{\sigma}$ until we find a term $x^{\sigma'}$ belonging to $\mathcal P_J$ such that $x^\gamma=x^\epsilon x^{\sigma'}$ with $\max(x^\eta x^\epsilon)\leq \min(x^{\sigma'})$ and we conclude as before because $x^\eta x^\epsilon <_{\mathrm{lex}}~x^{\sigma'}$ by \cite[Lemma~3.4(vi)]{Quot} and $\max(x^\eta)\leq \min(x^\epsilon)$ by construction. 
\end{proof}

\begin{proposition}\label{prop:NoethConfl}
The relation $\rid{\mathcal B^{(t)}}$ is Noetherian and confluent.
\end{proposition}

\begin{proof}
We show that the rewriting procedure given by the relation $\rid{\mathcal B^{(t)}}$ ends after a finite number of steps when applied to any term. 

Let $x^\delta x^\alpha$ any term  that belongs to $J\setminus \partial\mathcal O$, with $x^\alpha\in \partial\mathcal O$. Then there exist a term $x^\gamma \in \partial\mathcal O$ and a term $x^\eta$ such that $\max(x^\eta)\leq \min(x^\gamma)$ and $x^\delta x^\alpha=x^\eta x^\gamma$.

Hence the first step of the rewriting procedure consists in computing $x^\delta x^\alpha - x^\eta b_{\gamma}$, in which every term that appears with a non-null coefficient is of type $x^\eta m$, with $m\in \mathcal N(J)$. If $x^\eta m$ belongs to $J\setminus \partial\mathcal O$ then we can apply Lemma \ref{lemma:lex} and conclude.

For confluency, it is now enough to observe that for every $g \in R$, by Noetherianity there exists  $h\in \langle \mathcal N(J)\rangle$ such that $g\rid{\mathcal B^{(t)}} h$. Then $h-\mathrm{Nf}_{I}(g)\in I\cap \langle \mathcal N(J)\rangle$, but the latter is $\{0\}$ since $I$ is generated by the marked basis $\mathcal H$. Hence $h=\mathrm{Nf}_{I}(g)$.
\end{proof}
  
\begin{remark}
In terms of reduction structures \cite{CMR}, $\mathcal B$ and the terms that allow the definition of $\mathcal B^{(t)}$ give a substructure of the border reduction structure considered in~\cite{KRvol2}. The difference is that in \cite{KRvol2} the authors admit multiplication of polynomials in $\mathcal B$ by any term. The polynomial reduction of Definition \ref{def:riscrittura border}, being Noetherian, proves that the border reduction structure of \cite{KRvol2} is weakly Noetherian \cite[Definition~5.1]{CMR}.
\end{remark}

\begin{proposition}\label{prop:decomposition}
If $\mathcal H$ is a $\mathcal P_J$-marked basis then, for every integer~$t$, for every $b_{\tau}\in \mathcal B_{t-1}$ and for every non-multiplicative variable $x_i$ of $x^\tau$, we have the following decompositions, where $L_b$ is a linear form of multiplicative variables of the head term of $b$ and $c_h\in \mathbb K$:
\begin{equation}\label{eq:syzygy 1} 
x_ib_{\tau}= \sum_{b\in \mathcal B_{t-1}} L_b b + \sum_{h\in \mathcal H_t} c_h h,
\end{equation} 
if $x_ix^\tau=x_k x^\gamma$, with $x_k$ multiplicative variable of $x^\gamma\in \partial\mathcal O$,
\begin{equation}\label{eq:syzygy 2} 
x_ib_{\tau}-x_kb_{\gamma}= \sum_{b\in \mathcal B_{t-1}} L_b b + \sum_{h\in \mathcal H_t} c_h h,
\end{equation} 
if $x_ix^\tau=x_k x^\gamma$, with $x_k$ non-multiplicative variable of $x^\gamma\in \partial\mathcal O\setminus\{x^\tau\}$.
\end{proposition}

\begin{proof}
If $\mathcal H$ is a $\mathcal P_J$-marked basis then $f \crid{\mathcal B^{(t)}} 0$ for every polynomial $f\in I$ because $I_t\cap \langle \mathcal N(J)_t \rangle_A=0$, for every $t\leq \mathrm{reg}(J)+1$.

Let $p:=x_ib_\tau - x_kb_{\gamma}$ in both cases that are considered in the statement. 

If $x_ix^\tau=x_k x^\gamma$ with $x_k$ multiplicative variable of $x^\gamma$ then $x_ib_\tau \rid{\mathcal B^{(t)}} p=x_ib_\tau - x_kb_{\gamma}$. 

We now observe that in any case the terms appearing with a non-null coefficient in the polynomial $p=x_ib_\tau - x_kb_{\gamma}$ are multiples of a term in $\mathcal N(J)$ by a variable, i.e. they are of type $x_\ell m$, with $m\in \mathcal N(J)$. Hence, they belong to either $\mathcal N(J)$ or $\partial\mathcal O_t$. 

If $x_\ell m$ belongs to $\mathcal N(J)$ then $x_\ell m \crid{\mathcal B^{(t)}} x_\ell m$.

If $x_\ell m$ belongs to $\mathcal P_J$, then  there is $\mu\in\mathbb{K}$ such that $p \rid{\mathcal B^{(t)}} p- \mu \ b_{x_\ell m}$. 

Otherwise, if $x_\ell m$ belongs to $\partial\mathcal O \setminus\mathcal P_J$ then $x_\ell m=x_\ell x_h m'$, where $x_h=\min(x_\ell m)$, $m'= m/x_h \in \mathcal N(J)$ and $x_\ell m'\in \partial\mathcal O$. Hence, $p \rid{\mathcal B^{(t)}} p- \mu \ x_hb_{x_\ell m'}$ for some $\mu\in \mathbb K$, because $x_h$ is multiplicative variable of $x_\ell m'$ .

We can conclude by repeating the above argument.
\end{proof}

The following result is a version of \cite[Definition 20 and Proposition 21]{KK2005} in terms of multiplicative and non-multiplicative variables (see \cite{HM2011} for a careful study of syzygies of $\partial \mathcal O$-marked bases).

\begin{lemma}\label{lemma: syzygies generators}
The couples of distinct terms $x^\tau$ and $x^\gamma$ in $\partial\mathcal O$ such that either $x_ix^\tau=x^\gamma\in \mathcal P_J$ or $x_ix^\tau=x_kx^\gamma=x_ix_km$, with $m\in\mathcal N(J)$ and either $x_i$ non-multiplicative of $x^\tau$ or $x_k$ non-multiplicative of $x^\gamma$, give rise to a set of syzygies of type $[\dots,x_i,\dots,-1,\dots]$ and $[\dots,x_i,\dots,-x_k,\dots]$, respectively, of $\partial\mathcal O$ which generate the first module of syzygies of $\partial\mathcal O$. 
\end{lemma}

\begin{proof}
The statement holds without the request on the variables thanks to \cite[Lemma 13.1 and Corollary 13.2]{MMM}, because $\partial\mathcal O$ is a Gr\"obner basis for $J$, being $J$ a monomial ideal.
Then it is enough to observe that $x_i$ and $x_k$ cannot be both multiplicative otherwise they both should coincide with $\min(x_ix_km)$. 
\end{proof}

\begin{theorem}\label{th:dependence}
A polynomial $h_\beta$ of $\mathcal H$ depends on $\mathcal H\setminus\{h_\beta\}$ if and only if $h_\beta$ appears with a non-null constant coefficient in a representation of type \eqref{eq:syzygy 1} or \eqref{eq:syzygy 2}.
\end{theorem}

\begin{proof}
We recall that the family of all $\partial\mathcal O$-marked bases is flat. This fact can be proved observing for example that the family of all $\partial\mathcal O$-marked bases is isomorphic to the family of $\mathcal P_J$-marked bases (see \cite[Corollary 2]{BC2022}), which is flat.

Letting $\mathcal B=\partial\mathcal O$, the syzygies arising from \eqref{eq:syzygy 1} and \eqref{eq:syzygy 2} generate the first module of syzygies of the border (see Lemma \ref{lemma: syzygies generators}). Then, for any $\partial\mathcal O$-border basis $\mathcal B$, the corresponding syzygies of $\mathcal B$ that are obtained from them by $\rid{\mathcal B^{(t)}}$ generate the first module of syzygies of $\mathcal B$, thanks to the criterion of Artin for flat morphisms (see \cite[Corollary to Proposition 3.1]{Artin}). 

Now we can conclude observing that in a representation of type \eqref{eq:syzygy 1} and \eqref{eq:syzygy 2} only proper multiples of polynomials of $\mathcal B_{t-1}$ and polynomials of $\mathcal H_t$ appear. 
\end{proof}

\begin{proposition}\label{prop: recursion}
Let $I$ be the ideal generated by a $\mathcal P_J$-marked basis $\mathcal H$. Let $x^\sigma$ be a term in $\partial\mathcal O_{t-1}$ and $x_i$ a variable. 
If $\mathrm{Nf}_{I}(x^\sigma) = \sum_{m_j \in \mathcal N(J)_{t-1}} c_j m_j$, then
\begin{equation}\label{eq: recursion} 
\mathrm{Nf}_{I}(x_i x^\sigma) = \sum_{m_j \in \mathcal N(J)_{t-1}} c_j \mathrm{Nf}_{I}(x_i {m_j}).
\end{equation}
\end{proposition}

\begin{proof}
By construction, the polynomial $x^\sigma - \mathrm{Nf}_{I}(x^\sigma) = x^\sigma - \sum_{m_j \in \mathcal N(J)_{t-1}} c_j m_j$ belongs to $I$, hence $x_i x^\sigma - \sum_{m_j \in \mathcal N(J)_{t-1}} c_j x_i m_j$ belongs to $I$, as well. 

Since $\sum_{m_j \in \mathcal N(J)_{t-1}} c_jx_i m_j - \sum_{m_j \in \mathcal N(J)_{t-1}} c_j \mathrm{Nf}_I(x_i m_j)$ belongs to $I$ too, we have
$$x_i x^\sigma - \sum_{m_j \in \mathcal N(J)_{t-1}} c_j \mathrm{Nf}_I(x_i m_j) \in I.$$
Being $x_i x^\sigma - \mathrm{Nf}_{I}(x_i x^\sigma)$ a polynomial of $I$, then $\mathrm{Nf}_{I}(x_i x^\sigma)=\sum_{m_j \in \mathcal N(J)_{t-1}} c_j \mathrm{Nf}_{I}(x_i m_j)$ because $I_t\cap \langle \mathcal N(J)_t \rangle_A=0$     and we conclude.
\end{proof}

\begin{corollary}\label{cor: recursive construction}
Formula \eqref{eq: recursion} allows a recursive computation of the polynomials in $\mathcal B_t \setminus \mathcal H$, for every $t$, knowing the polynomials of $\mathcal H_t$.
\end{corollary}

\begin{proof}
With the same notation of Proposition \ref{prop:decomposition}, 
if $x_i {m_j}$ belongs to $\partial \mathcal O\setminus\mathcal P_J$, then $x_i m_j=x_kx^\alpha$ for a suitable term $x^\alpha\in \partial\mathcal O$ and a variable $x_k<_{\mathrm{lex}} x_i$, thanks to Lemma \ref{lemma: multiples border}.

Hence, the polynomials of $\mathcal B_t$ with head term of type $x_0x^\sigma$ must belong to $\mathcal H$. For the polynomials of $\mathcal B_t$ with head term of type $x_1x^\sigma$, we can apply the formula of Proposition \ref{prop:decomposition} in which only polynomials of $\mathcal B_t$ with head term divisible by $x_0$ are involved. And so on.
\end{proof}

\begin{remark}
In \cite[Appendix]{MMM}, the analog of formula \eqref{eq: recursion} for Gr\"obner bases is called a FGLM-formula (see \cite{FGLM}) and its use for the computation of the polynomials in a border basis is reported referring to a discussion with Marie-Fran\c{c}oise Roy.
\end{remark}

For every $t$ higher than the initial degree of $I$ and lower than or equal to $\mathrm{reg}(J)+1$, we consider the matrix $M_t$ constructed in the following way.

A first block of rows correspond to the terms of degree $t$ that are multiplicative multiples of the terms in $\partial\mathcal O_{t-1}$ and to the terms of degree $t$ in $\mathcal P_J$, in increasing order with respect to the degrevlex order. A second block of rows correspond to the terms of $\mathcal N(J)_t$. 

The columns of $M_t$ are arranged in three blocks. The first block is made by the vectors of the coefficients of the polynomials that are multiples of the polynomials of $\mathcal B_{t-1}$ by a multiplicative variable of the corresponding head term and by the vectors of the coefficients of the polynomials of $\mathcal H_t$, ordered so that their head terms are in increasing degrevlex order (like the first block of rows). The second block are the vectors of the coefficients of the polynomials $x_ib_\tau$ that fit the case \eqref{eq:syzygy 1}. The third block of columns are the vectors of the coefficients of the polynomials $x_ib_{\tau}-x_kb_{\gamma}$ that fit the case~\eqref{eq:syzygy 2}.

It is noteworthy that, thanks to the features of quasi-stable ideals, the columns in the first block have the coefficient $1$ of the head term on pairwise different rows, so that these coefficients will become the pivots of the first block of columns in the completely reduced form of~$M_t$ and any division can be avoided in the performance of the reduction.

\begin{remark}
If we denote by $a:=\min\{t : I_t\not=0\}$ the initial degree of $I$, the matrix $M_a$ is made of the coefficients of the polynomials of degree $a$ of the $\mathcal P_J$-marked basis of $I$ only, which are independent by construction. For this reason we do not need to consider $M_a$. 
\end{remark}

\begin{remark}\label{rem: CI comp cost}
For every degree $t$, both the number of rows and the number of columns of a matrix $M_t$ are of order $O((n+1)^2\ \vert\mathcal N(J)_{t-2}\vert)$. 
\end{remark}

\begin{example}\label{ex: matrices 1}
Let us take the Artinian quasi-stable ideal $J:=(x_1^2,x_0^2)\subseteq \kk[x_0,x_1]$ with $x_0<x_1$. The following polynomials 

$h_1= x_0^2 + d_1 x_0x_1$,

$h_2= x_1^2 + d_2 x_0x_1$,

$h_3 = x_0^2x_1$

\noindent form a $\mathcal P_J$-marked basis, for every choice of the value of the parameters $d_1,d_2$.
The multiples of the polynomials $h_1$ and $h_2$ by the multiplicative variables of the respective head terms are:
\vskip 1mm

$x_0h_1=  x_0^3 + d_1 x_0^2x_1$,

$x_0h_2=  x_0x_1^2 + d_2 x_0^2x_1$,

$x_1h_2=  x_1^3 + d_2 x_0x_1^2$

\noindent and the multiples of the polynomials $h_1$ and $h_2$ by the non-multiplicative variables of the respective head terms are:

$x_1h_1 = x_0^2x_1 + d_1 x_0x_1^2.$

\vskip 1mm
\noindent Here is the matrix $M_3$, where the first row gives labels to the columns by the corresponding polynomials and the first column gives labels to the rows by the corresponding terms. The second block of rows is empty because $\mathcal N(J)_3=\emptyset$. Also the third block of columns is empty. 

\begin{equation}\label{eq:Mt}
\begin{blockarray}{cccccc}
 {\scriptstyle x_0h_1 }&{\scriptstyle h_3} &{\scriptstyle  x_0h_2} &{\scriptstyle  x_1 h_2} &{\scriptstyle x_1h_1}\\ 
\begin{block}{(cccc|c)c}
  1 & 0 & 0 & 0 & 0 &{\scriptstyle x_0^3}\\
d_1 &1 & d_2 & 0 & 1 & {\scriptstyle x_0^2x_1}\\
 0 &0 & 1 & d_2 & d_1& {\scriptstyle x_0x_1^2}\\
0 &0 &0 &1 & 0 & {\scriptstyle x_1^3}\\
 \end{block}
\end{blockarray}
\end{equation}
\end{example}

The following result follows by standard linear algebra.

\begin{proposition}\label{prop:elim}
Let $\mathcal H$ be a $\mathcal P_J$-marked basis, $M_t$ the matrix constructed above for a certain degree $t$ and $M'_t$ the complete reduced matrix of $M_t$. 

A polynomial $h_\beta$ of degree $t$ depends on $\mathcal H\setminus \{h_\beta\}$ if and only if there is a non-null element in at least a crossing of one of the columns of $M'_t$ of the second or third block and the row corresponding to the head term of $h_\beta$.
\end{proposition}

\begin{corollary}\label{cor: CI}
Let $J$ be an Artinian monomial ideal. The strict complete intersection locus in $\MFScheme{J}(\kk)$ is an open subset ${\bf CI}_J$ of $\MFScheme{J}$.
\end{corollary}

\begin{proof} 
Let $\mathscr U\subseteq \kk[C]$ be the ideal defining the scheme $\MFScheme{J}$ and $\mathscr H\subseteq R_{\kk[C]}$ be the $\mathcal P_J$-marked basis modulo $\mathscr U$ as decribed in Section \ref{sec:marked functors}. For every $t\leq \mathrm{reg}(J)$, the polynomials in the corresponding $\partial\mathcal O$-marked basis $\mathcal B_t$ can be computed by $\mathscr H$, like described in Corollary \ref{cor: recursive construction}. Hence, the elements of the matrices $M_t$ are polynomials in the parameters $C$. 

Let $M'_t$ be the completely reduced form of $M_t$, which can be obtained without divisions, thanks to the shape of the matrix $M_t$. Thus, the open subset ${\bf CI}_J$ of $\MFScheme{J}$ parametrizing the complete intersections is defined by the non-vanishing of the elements corresponding to polynomials of $\mathcal H$ in the second or third block of columns of $M'_t$, in every combination and number that is sufficient in order to have no more than $c$ minimal generators for the ideal generated by $\mathscr H$, modulo~$\mathscr U$.
\end{proof}

\begin{corollary}\label{cor:open CI}
The strict complete intersection locus in a Hilbert scheme with a non-constant Hilbert polynomial is an open subset. 
\end{corollary}

\begin{proof}
Thanks to Corollary \ref{cor: CI}, it is enough to argue analogously to the proof of Corollary \ref{cor:open Gore}.
\end{proof}

\begin{example} \label{ex: matrices 2}
We go back to Example \ref{ex: matrices 1} where we  considered the Artinian quasi-stable ideal $J:=(x_1^2,x_0^2)\subseteq \kk[x_0,x_1]$ with $x_0<x_1$. The ideal $J$ is a complete intersection. Hence, the subscheme of $\MFScheme{J}(\mathbb K)$ parameterizing the complete intersections defined by a $\mathcal P_J$-marked basis is non-empty. In order to apply our strategy to compute such subscheme, we now continue to study the matrix $M_3$ already constructed in the previous example.

By a complete reduction process, starting from the matrix $M_3$ as in \eqref{eq:Mt} we get:

$$
\left(\begin{array}{cccc|c}
  1 &{0} & 0 & 0 &\tikzmarknode{n3}{0} \\
  \tikzmarknode{n1}{0} &{1} & 0 & 0 & \tikzmarknode{n2}{1-d_1d_2} \\
 0 &{0} & 1 & 0 & d_1  \\
 0 &{0} &0 &1 & \tikzmarknode{n4}{0}  \\
\end{array}\right)$$
\begin{tikzpicture}[remember picture,overlay,
  highlight/.style={draw=blue,rounded corners, thick, 
                    minimum width=1.5em,minimum height=2.5ex},
  note/.style={font=\small,red,text width=3cm},
  ]
%  \node[highlight] at ([yshift=1ex]pic cs:n1) {};
%  \node[note,left=0.5cm,anchor=east] at (pic cs:n1) 
%                      {Element $(1,1)$ of matrix equals 1};
  \draw[highlight] ([yshift=2ex,xshift=2em]pic cs:n3) 
                   rectangle ([yshift=-1ex,xshift=-2em]pic cs:n4) ;
  \draw[highlight] ([yshift=-0.8ex,xshift=-0.8em]pic cs:n1) 
                   rectangle ([yshift=2ex,xshift=2em]pic cs:n2) ;
\end{tikzpicture}
We highlight the second row, which corresponds to the head term of $h_3$, and the last column, which belongs to the second block of columns of $M_3$. We focus on the element $1-d_1d_2$. By Proposition \ref{prop:elim}, if this element is non-null, then $h_3$ depends on $\mathcal H\setminus\{h_3\}$. Hence, from the last column we obtain $x_1h_1=(1-d_1d_2) h_3+d_1x_0h_2$ and so the syzygy $[-x_2,d_1x_1,1-d_1d_2]$. Thus $h_3$ is dependent if and only if $1-d_1d_2\not=0$, i.e.~an ideal generated by a $\mathcal P_J$-marked basis of the given type is a complete intersection if and only if $(1-d_1d_2) \not=~0$.
Note that in this case, being the codimension~$2$, the property to be a complete intersection is equivalent to the property to be Gorenstein.
\end{example}

\begin{remark}
In Examples \ref{ex: matrices 1} and \ref{ex: matrices 2} the matrix $M_3$ does not have rows corresponding to terms in $\mathcal N(J)$ because $\mathcal N(J)_3$ is empty. 

However, to our aims we do not need the rows corresponding to terms of $\mathcal N(J)$ because the non-null constants which we are interested in are elements of other rows, in which the pivots appear, by construction and by Proposition \ref{prop:decomposition}. Hence, the second block of rows can be avoided in the construction of the matrices $M_t$, for every $t$.
\end{remark}

\section{Construction of complete intersections inside a given ideal}
\label{sec:CI}

In this section, we focus on the problem of constructing a complete intersection contained in a given polynomial ideal $I$, with the same height of $I$, even if $I$ is not a complete intersection. Indeed, every ideal $I$ in a Noetherian ring has a set of generators containing a complete intersection with the same height of $I$, even if $I$ is not a complete intersection. 

A classical proof of this statement is in \cite[Chapter VI, Proposition 3.5]{Kunz}, for example, but it does not give an efficient constructive method. The problem is that the knowledge of a primary decomposition is required. For this reason, it is noteworthy that very recently an efficient method to recognize complete intersections has been given in the paper \cite{HAPS}.

In \cite[Section 1]{EiStu1994} the authors gave a computational method to construct a regular sequence in $I$ of length equal to the codimension of $I$, for every polynomial ideal~$I$. However, this method requires a further step to check that some suitable given polynomials are a regular sequence. 

In \cite[Proposition 5.1]{Seiler2012} the author shows that every Pommaret (Gr\"obner) basis contains a complete intersection with the same height of the ideal that the Pommaret basis generates, explicitly exhibiting the complete intersection without the necessity of any computation. Indeed, the polynomials that generate the complete intersection are those in the Pommaret basis with initial term equal to one of the powers of the variables contained in the Pommaret basis of the initial ideal. Even in this case the properties of Gr\"obner bases have a crucial role. Is there an analogous result for marked bases over a quasi-stable ideal?
The following example shows that an analogous result does not hold for marked bases.

\begin{example}\label{ex:example one}
Consider the quasi-stable ideal $J=(x_2^2,x_1^2)\subseteq \mathbb K[x_0,x_1,x_2]$ with $x_0<x_1<x_2$ and the $\mathcal P_J$-marked basis
$\mathcal H=\{h_1:=x_2^2 - x_2x_1 - x_2x_0 + x_1x_0, h_2:=x^2_1 - x_2x_1, h_3:=x_2x_1^2 - 2x_2x_1x_0 + 2x_1x_0^2 \}$. We have $(x_2+x_0)h_2=-x_1h_1 \in (h_1)$, hence $h_2$ is a zero-divisor on $\mathbb K[x_0,x_1,x_2]/(h_1)$ and so $h_1,h_2$ is not a regular sequence, although their head terms are powers of variables. Note that $\mathcal H$ is not a Gr\"obner basis with respect to any term order because $x_2x_1>x_1^2$ and the polynomial $h_2$ is marked on the term $x_1^2$. 
However, $h_1,h_3$ is a regular sequence because $((h_1):(h_3))=(h_1)$ and the ideal generated by $\mathcal H$ is not a complete intersection. Since it has codimension~$2$, it is also not Gorenstein, indeed, thanks to a result of Serre (see \cite{Huneke} and the references therein). 
%\todo[inline]{however, $h_1,h_3$ are not sufficient to generate $I$ ($I$ is not a complete intersection)-C. OK-F}

On the other hand, the different $\mathcal P_J$-marked basis $\mathcal H'=\{h'_1:=x_2^2 - 1/2 x_2x_1 - x_2x_0, h'_2:=x_1^2 - 1/2x_2x_1 - x_1x_0,
h'_3:=x_2x_1^2 - 2x_2x_1x_0\}$ satisfies the expectation that $h'_1,h'_2$ is a regular sequence. In this case the ideal generated by $\mathcal H'$ is a complete intersection and coincides with the ideal generated by $h'_1$, $h'_2$.  
%\todo[inline]{is this a complete intersection? I think so}
\end{example}

\begin{example}\label{ex: other situation}
We can also have the following situation. Given the quasi-stable ideal $J=(x_3^2, x_2x_3,x_1^2x_3,x_2^4) \subseteq \kk[x_0,x_1,x_2x_3]$ with $x_0<\dots<x_3$, consider the $\mathcal P_{J}$-marked basis $\mathcal H=\{h_1:=x_3^2, h_2:=x_2x_3+x_2^2+2x_1x_2+x_1^2, h_3:=x_1^2x_3-x_2^3-4x_1x_2^2-5x_1^2x_2-2x_1^3, h_4:= x_2^4+4x_1x_2^3+6x_1^2x_2^2+4x_1^3x_2+x_1^4\}$. The polynomials $h_1$ and $h_4$ give the expected regular sequence, but the ideal generated by $\mathcal H$ is the complete intersection generated by the regular sequence $h_1, h_2$. 
\end{example}

Example \ref{ex:example one} does not exclude that a marked basis contains a regular sequence with the length equal to the codimension of the polynomial ideal $I$, even when this regular sequence is not the expected one. 

Hence, the following questions arise: Given a marked basis $\mathcal H$, when does $\mathcal H$ contain a regular sequence of length equal to the height of $(\mathcal H)$ and how can we compute it? Is there a method to compute  a regular sequence of length equal to the height of $(\mathcal H)$ which is contained in $(\mathcal H)$?

We will give a qualitative answer to the above questions in terms of marked bases adapting \cite[Theorem 1.3]{EiStu1994} to the setting of marked bases, instead of Gr\"obner bases. 

{\em From now}, let $J\subseteq R=\mathbb K[x_0,\dots,x_n]$ be a quasi-stable ideal of codimension $c$ and $\mathcal H=\{h_1,\dots,h_t\}$ be a $\mathcal P_J$-marked basis. Let $I$ be the ideal generated by $\mathcal H$. 

\begin{proposition}\label{prop:EiStu proposition}
\cite[Proposition 1.4]{EiStu1994}
Over an infinite field $\kk$, let $\mathcal F_1,\dots,\mathcal F_c \subset R$ be sets of polynomials such that, for every subset $U \subset \{ 1,\dots , c\}$, the set of polynomials $\cup_{i\in U} \mathcal F_i$ generates an ideal of codimension $\geq \vert U \vert$. Then the polynomials
$$f_1=\sum_{f\in \mathcal F_1} r_{1,f} f, \dots, f_c=\sum_{f\in \mathcal F_c} r_{c,f} f$$ 
generate an ideal of codimension $c$ in $R$, for every values of $r_{i,f}$ varying in a suitable non-empty Zariski open subset.
\end{proposition}

\begin{theorem}\label{theorem:EiStu theorem}
There exists a non-empty subset $\mathcal K\subseteq \mathcal{H}$ and a partition $\mathcal K=\mathcal K_1\cup\dots\cup \mathcal K_c$ into non-empty subsets, such that the polynomials 
\begin{equation}\label{eq:regular sequence}
f_1=\sum_{f\in \mathcal K_1} r_{1,f} f, \dots, f_c=\sum_{f\in \mathcal K_c} r_{c,f} f
\end{equation} 
generate an ideal of codimension $c$ in $R$ for values of $r_{i,f}$ varying in a suitable Zariski open subset.
\end{theorem}

\begin{proof}
Consider the subset $\mathcal F=\{\tau \in \mathcal P_J : \min(\tau) \geq n-c+1\}\subseteq \mathcal P_J$. By construction $\mathcal F$ is the Pommaret basis of the quasi-stable ideal $(\mathcal F)\subseteq J$. Then, the set 
$\mathcal K:=\{h\in \mathcal H : \mathrm{Ht}(h)\in \mathcal F\}$ is a $\mathcal F$-marked set and the codimension of the ideal generated by $\mathcal K$ is higher than or equal to the codimension of the ideal generated by $\mathcal F$ thanks to Proposition \ref{prop:known facts}(i). Observe that the codimension of $(\mathcal F)$ is equal to $c$, by construction.

Following \cite[Remark 5.2]{Seiler2012}, let $\mathcal F_j:=\{\tau\in \mathcal P_J : \min(\tau)=x_{n-j+1}\} \subset (x_{n-j+1})$ and $\mathcal K_j:=\{h\in \mathcal H : \mathrm{Ht}(h)\in \mathcal F_j\}$, for every $j\in \{1,\dots,c\}$. 

Now, we have a partition $\mathcal F=\mathcal F_1\cup\dots\cup \mathcal F_c$ where the sets $\mathcal F_j$ satisfy the hypothesis of Proposition \ref{prop:EiStu proposition}. Indeed, for every subset $U \subset \{ 1,\dots , c\}$, the set of terms $\cup_{i\in U} \mathcal F_i$ generates an ideal of codimension $\geq \vert U \vert$ thanks to the fact that terms of types $x_{n-c+i}^{\alpha_i}$ belong to $\cup_{i\in U} \mathcal F_i$, for every $i\in U$.

For every subset $U \subset \{ 1,\dots , c\}$ consider the set of variables 
$$X_U:=\{x_{n-c+j} : j\not\in U \text{ and } j\geq c\}.$$ Then the image $\cup_{i\in U} \mathcal F_i$ in $R/(X_U)$ is the Pommaret basis of the ideal $J_U$ it generates and the image of $\cup_{i\in U} \mathcal K_i$ is a marked set on this Pommaret basis. Hence, thanks to Proposition \ref{prop:known facts}, the codimension of the ideal $(\cup_{i\in U} \mathcal K_i)$ is higher than or equal to the codimension of $(J_U)$, that is $\vert U\vert$ by construction (see \cite[Chapter 9, Section 1, Proposition 3]{CLO} for an easy computation of the codimension of a monomial ideal).

In conclusion, even the partition $\mathcal K=\mathcal K_1\cup\dots\cup \mathcal K_c$ satisfies the hypothesis of Proposition \ref{prop:EiStu proposition} and the thesis is proved. 
\end{proof}

\begin{remark}\label{rem:corollary:due}
It is clear that a $\mathcal P_J$-marked basis $\mathcal H$ contains a complete intersection of codimension~$c$ if Theorem \ref{theorem:EiStu theorem} gives a regular sequence of type \eqref{eq:regular sequence} such that every coefficient in each polynomial $f_i$ is null, except one. 
\end{remark}

\begin{remark}\label{rem:uno}
The partition $\mathcal K=\mathcal K_1\cup\dots\cup \mathcal K_c$ that has been constructed in the proof of Theorem \ref{theorem:EiStu theorem} satisfies the property that $\cup_{i \leq s} \mathcal K_i$ is a marked set which generates an ideal of codimension~$\geq s$, for every $1\leq s< c$, and equal to $c$ if $s=c$.
\end{remark}

In conclusion, together with the application of one of the already known methods to check if a sequence of polynomials is a regular sequence, the above results provide an algorithm to compute a regular sequence in $I$ with the same height of $I$, starting from a marked basis of $I$. This algorithm is an adaptation to marked bases of an algorithm obtained for Gr\"obner bases in the paper \cite{EiStu1994}. 

%Up to our knowledge, at the moment methods to check if a sequence of polynomials is a regular sequence require primary decomposition of ideals, the computation of which can be heavy. Analogously, 
Although from the proof of \cite[Proposition 1.4]{EiStu1994} it can be deduced that a description of the non-empty Zariski open subset involved in Theorem~\ref{theorem:EiStu theorem} should need irreducible decomposition of varieties and thus is computationally expensive, some examples can be worked out.

\begin{example}
Consider the quasi-stable ideal $J=(x_3^3, x_2x_3^2,$ $x_2^2x_3, x_1x_3^2, x_1x_2x_3,$ $x_0x_3^2, x_1^2x_3, x_0x_2x_3, x_2^4)\subseteq \mathbb R[x_0,\dots,x_3]$ with $x_0<\dots<x_3$ and the ideal $I$ generated by the following $\mathcal P_J$- marked basis

$\mathcal H=\{ x_3^3,  \ x_2x_3^2,  \ x_2^2x_3 + x_2^3 + 2x_1x_2^2 + x_1^2x_2, \ x_1x_3^2,  \ x_1x_2x_3 + x_1x_2^2 + 2x_1^2x_2 + x_1^3,$ 

$\ x_0x_3^2, \ x_1^2x_3 - x_2^3 - 4x_1x_2^2 - 5x_1^2x_2 - 2x_1^3,
\ x_0x_2x_3 + x_0x_2^2 + 2x_0x_1x_2 + x_0x_1^2,$

$\ x_2^4 + 4x_1x_2^3 + 6x_1^2x_2^2 + 4x_1^3x_2 + x_1^4
\}.$

\noindent With the notation introduced in the proof of Theorem \ref{theorem:EiStu theorem}, we have:

$\mathcal K=\{h_1=x_3^3,  \ h_2=x_2x_3^2,  \ h_3=x_2^2x_3 + x_2^3 + 2x_1x_2^2 + x_1^2x_2,  h_4=x_2^4 + 4x_1x_2^3 + 6x_1^2x_2^2 + 4x_1^3x_2 + x_1^4\}$

$\mathcal K_1=\{x_3^3\},\quad \mathcal K_2=\{x_2x_3^2,  \ x_2^2x_3 + x_2^3 + 2x_1x_2^2 + x_1^2x_2,  x_2^4 + 4x_1x_2^3 + 6x_1^2x_2^2 + 4x_1^3x_2 + x_1^4\}$.

\noindent Like we already observed in Remark \ref{rem:uno}, $\mathcal K=\mathcal K_1 \cup \mathcal K_2$ is a marked set on the quasi-stable ideal $\bar J=(x_3^3, x_2x_3^2,  x_2^2x_3, x_2^4)$, but it is not a marked basis. Indeed, a reduced form of $x_3\cdot h_3$ by $\mathcal K$ is not $0$, but $x_1^2x_2x_3-4x_1^2x_2^2-2x_1^3x_2$.

Moreover, we highlight that if a sequence of polynomials of $I$ is a regular sequence, then it is not necessarily characterized by the shape of the polynomials given in \eqref{eq:regular sequence}. For example, $h_1+h_2, h_4$ is a regular sequence of $I$ of length equal to the codimension of $I$, but it does not have the shape of \eqref{eq:regular sequence}. 
\end{example}

\section*{Acknowledgments}
The first and second authors are members of GNSAGA (INdAM, Italy). The first author was partially supported by the project \lq\lq Metodi matematici nelle scienze computazionali\rq\rq (Dipartimento di Matematica, University of Turin) and with the second author by the Institute of Mathematics of the University of Kassel. The authors thank
the Institute of Mathematics of the University of Kassel and 
Università degli Studi di Napoli Federico II for their hospitality during the preparation of this paper.

\appendix
\section{Data for Example \ref{ex:Gore 2}}
The equations defining the marked scheme of $J$ in Example \ref{ex:Gore 2} are the following ones, where the $d_{i,j}$s are the coefficients of the marked polynomials of the set $\mathcal H$ listed in~\eqref{eq:exNoGore2MS}:
\begin{equation}\label{eq:exGore2MF}
\begin{gathered}
d_{{1,1}}d_{{2,1}}d_{{4,1}}+d_{{1,1}}d_{{2,2}}d_{{5,1}}-d_{{1,1}}d_{{2,3}}-d_{{1,3}}d_{{6,1}}+d_{{1,2}}-d_{{4,1}}=0,\\
d_{{1,3}}d_{{2,1}}d_{{3,1}}d_{{4,1}}+d_{{1,3}}d_{{2,2}}d_{{3,1}}d_{{5,1}}-d_{{1,2}}d_{{3,1}}d_{{4,1}}-d_{{1,2}}d_{{3,2}}d_{{5,1}}+d_{{1,3}}d_{{2,3}}d_{{3,1}}+\\
-d_{{1,3}}d_{{3,3}}d_{{6,1}}-d_{{1,2}}d_{{3,3}}-d_{{1,3}}d_{{3,2}}+d_{{1,1}}+d_{{5,1}}=0,\\
-d_{{2,1}}d_{{2,3}}d_{{3,1}}d_{{4,1}}-d_{{2,2}}d_{{2,3}}d_{{3,1}}d_{{5,1}}+d_{{2,2}}d_{{3,1}}d_{{4,1}}+d_{{2,2}}d_{{3,2}}d_{{5,1}}+
{d_{{2,3}}}^{2}d_{{3,1}}+\\
+d_{{2,3}}d_{{3,3}}d_{{6,1}}-d_{{2,2}}d_{{3,3}}-d_{{2,3}}d_{{3,2}}+d_{{2,1}}-d_{{6,1}}=0.
\end{gathered}
\end{equation}
The coefficient matrix of the system $\Sigma_{\mathcal H}$ is the following one, where we are omitting some zero rows:
\begin{equation}\label{eq:exNoGore2Matrix}    
 \left[\begin{array}{cccc}
d_{1,3} & 0 & 0 & 0 
\\
 -d_{1,1} d_{2,1}+1 & 0 & 0 & 0 
\\
 -d_{1,1} d_{2,2} & 0 & 0 & 0 
\\
 0 & d_{2,1} d_{4,1}+d_{2,2} d_{5,1}-d_{4,1} d_{6,1}+d_{2,3} & -d_{5,1} d_{6,1}+1 & 0 
\\
 -d_{1,3} d_{2,1} d_{3,1}-d_{1,2} d_{3,1} & 0 & 0 & 0 
\\
 -d_{1,3} d_{2,2} d_{3,1}-d_{1,2} d_{3,2}+1 & 0 & 0 & 0 
\\
 0 & \Delta_{7,2}& \Delta_{7,3} & 0 
\end{array}\right]
\end{equation}
with
\[
\Delta_{7,2}=-d_{2,1} d_{3,1} d_{4,1}^{2}-d_{2,2} d_{3,1} d_{4,1} d_{5,1}-d_{2,3} d_{3,1} d_{4,1}-d_{3,3} d_{4,1} d_{6,1}-d_{3,2} d_{4,1}+1,
\]
\[\Delta_{7,3}=-d_{2,1} d_{3,1} d_{4,1} d_{5,1}-d_{2,2} d_{3,1} d_{5,1}^{2}-d_{2,3} d_{3,1} d_{5,1}-d_{3,3} d_{5,1} d_{6,1}+d_{3,1} d_{4,1}+d_{3,3}.\]
The matrix \eqref{eq:exNoGore2Matrix} of the system $\Sigma_{\mathcal H}$ has the following five non-zero minors of order~$3$:
\small{
\begin{equation}\label{eq:exGore2NoGore}
\begin{split}
&\left(d_{1,1} d_{2,1}-1\right) \left(d_{2,1}^{2} d_{3,1} d_{4,1}^{2} d_{5,1}+2 d_{2,1} d_{2,2} d_{3,1} d_{4,1} d_{5,1}^{2}+d_{2,2}^{2} d_{3,1} d_{5,1}^{3}+2 d_{2,1} d_{2,3} d_{3,1} d_{4,1} d_{5,1}+\right.\\
&d_{2,1} d_{3,3} d_{4,1} d_{5,1} d_{6,1}+2 d_{2,2} d_{2,3} d_{3,1} d_{5,1}^{2}+d_{2,2} d_{3,3} d_{5,1}^{2} d_{6,1}-2 d_{2,1} d_{3,1} d_{4,1}^{2}-2 d_{2,2} d_{3,1} d_{4,1} d_{5,1}+\\
&+d_{2,3}^{2} d_{3,1} d_{5,1}+d_{2,3} d_{3,3} d_{5,1} d_{6,1}+d_{3,1} d_{4,1}^{2} d_{6,1}+d_{3,2} d_{4,1} d_{5,1} d_{6,1}-d_{2,1} d_{3,3} d_{4,1}+\\
&\left.-d_{2,2} d_{3,3} d_{5,1}-2 d_{2,3} d_{3,1} d_{4,1}-d_{2,3} d_{3,3}-d_{3,2} d_{4,1}-d_{5,1} d_{6,1}+1\right)
, \\
&d_{1,1} d_{2,2} \left(d_{2,1}^{2} d_{3,1} d_{4,1}^{2} d_{5,1}+2 d_{2,1} d_{2,2} d_{3,1} d_{4,1} d_{5,1}^{2}+d_{2,2}^{2} d_{3,1} d_{5,1}^{3}+2 d_{2,1} d_{2,3} d_{3,1} d_{4,1} d_{5,1}+\right.\\
&+d_{2,1} d_{3,3} d_{4,1} d_{5,1} d_{6,1}+2 d_{2,2} d_{2,3} d_{3,1} d_{5,1}^{2}+d_{2,2} d_{3,3} d_{5,1}^{2} d_{6,1}-2 d_{2,1} d_{3,1} d_{4,1}^{2}+\\
&-2 d_{2,2} d_{3,1} d_{4,1} d_{5,1}+d_{2,3}^{2} d_{3,1} d_{5,1}+d_{2,3} d_{3,3} d_{5,1} d_{6,1}+d_{3,1} d_{4,1}^{2} d_{6,1}+d_{3,2} d_{4,1} d_{5,1} d_{6,1}+\\
&\left.-d_{2,1} d_{3,3} d_{4,1}-d_{2,2} d_{3,3} d_{5,1}-2 d_{2,3} d_{3,1} d_{4,1}-d_{2,3} d_{3,3}-d_{3,2} d_{4,1}-d_{5,1} d_{6,1}+1\right)
, \\
&-d_{1,3} \left(d_{2,1}^{2} d_{3,1} d_{4,1}^{2} d_{5,1}+2 d_{2,1} d_{2,2} d_{3,1} d_{4,1} d_{5,1}^{2}+d_{2,2}^{2} d_{3,1} d_{5,1}^{3}+2 d_{2,1} d_{2,3} d_{3,1} d_{4,1} d_{5,1}+\right.\\
&+d_{2,1} d_{3,3} d_{4,1} d_{5,1} d_{6,1}+2 d_{2,2} d_{2,3} d_{3,1} d_{5,1}^{2}+d_{2,2} d_{3,3} d_{5,1}^{2} d_{6,1}-2 d_{2,1} d_{3,1} d_{4,1}^{2}+\\
&-2 d_{2,2} d_{3,1} d_{4,1} d_{5,1}+d_{2,3}^{2} d_{3,1} d_{5,1}+d_{2,3} d_{3,3} d_{5,1} d_{6,1}+d_{3,1} d_{4,1}^{2} d_{6,1}+d_{3,2} d_{4,1} d_{5,1} d_{6,1}+\\
&\left.-d_{2,1} d_{3,3} d_{4,1}-d_{2,2} d_{3,3} d_{5,1}-2 d_{2,3} d_{3,1} d_{4,1}-d_{2,3} d_{3,3}-d_{3,2} d_{4,1}-d_{5,1} d_{6,1}+1\right)
, \\
&-\left(d_{1,3} d_{2,2} d_{3,1}+d_{1,2} d_{3,2}-1\right) \left(d_{2,1}^{2} d_{3,1} d_{4,1}^{2} d_{5,1}+2 d_{2,1} d_{2,2} d_{3,1} d_{4,1} d_{5,1}^{2}+d_{2,2}^{2} d_{3,1} d_{5,1}^{3}+\right.\\
&+2 d_{2,1} d_{2,3} d_{3,1} d_{4,1} d_{5,1}+d_{2,1} d_{3,3} d_{4,1} d_{5,1} d_{6,1}+2 d_{2,2} d_{2,3} d_{3,1} d_{5,1}^{2}+d_{2,2} d_{3,3} d_{5,1}^{2} d_{6,1}+\\
&-2 d_{2,1} d_{3,1} d_{4,1}^{2}-2 d_{2,2} d_{3,1} d_{4,1} d_{5,1}+d_{2,3}^{2} d_{3,1} d_{5,1}+d_{2,3} d_{3,3} d_{5,1} d_{6,1}+d_{3,1} d_{4,1}^{2} d_{6,1}+\\
&\left.d_{3,2} d_{4,1} d_{5,1} d_{6,1}-d_{2,1} d_{3,3} d_{4,1}-d_{2,2} d_{3,3} d_{5,1}-2 d_{2,3} d_{3,1} d_{4,1}-d_{2,3} d_{3,3}-d_{3,2} d_{4,1}+\right.\\
&\left.-d_{5,1} d_{6,1}+1\right)
, \\
&-d_{3,1} \left(d_{1,3} d_{2,1}+d_{1,2}\right) \left(d_{2,1}^{2} d_{3,1} d_{4,1}^{2} d_{5,1}+2 d_{2,1} d_{2,2} d_{3,1} d_{4,1} d_{5,1}^{2}+d_{2,2}^{2} d_{3,1} d_{5,1}^{3}+\right.\\
&+2 d_{2,1} d_{2,3} d_{3,1} d_{4,1} d_{5,1}+d_{2,1} d_{3,3} d_{4,1} d_{5,1} d_{6,1}+2 d_{2,2} d_{2,3} d_{3,1} d_{5,1}^{2}+d_{2,2} d_{3,3} d_{5,1}^{2} d_{6,1}+\\
&-2 d_{2,1} d_{3,1} d_{4,1}^{2}-2 d_{2,2} d_{3,1} d_{4,1} d_{5,1}+d_{2,3}^{2} d_{3,1} d_{5,1}+d_{2,3} d_{3,3} d_{5,1} d_{6,1}+d_{3,1} d_{4,1}^{2} d_{6,1}+\\
&\left.d_{3,2} d_{4,1} d_{5,1} d_{6,1}-d_{2,1} d_{3,3} d_{4,1}-d_{2,2} d_{3,3} d_{5,1}-2 d_{2,3} d_{3,1} d_{4,1}-d_{2,3} d_{3,3}-d_{3,2} d_{4,1}+\right.\\
&\left.-d_{5,1} d_{6,1}+1\right).
\end{split}
\end{equation}
}

%\bibliographystyle{amsplain}
%\bibliography{CIGore}

\providecommand{\bysame}{\leavevmode\hbox to3em{\hrulefill}\thinspace}
\providecommand{\MR}{\relax\ifhmode\unskip\space\fi MR }
\providecommand{\MRhref}[2]{%
  \href{http://www.ams.org/mathscinet-getitem?mr=#1}{#2}
}
\providecommand{\href}[2]{#2}

\end{document}